\documentclass[11pt]{article}

\usepackage{amsmath,amssymb,amsthm}
\usepackage[latin1]{inputenc}
\usepackage{version,tabularx,multicol}
\usepackage{graphicx,float,psfrag}
\usepackage{stmaryrd}
 \usepackage{color}

\theoremstyle{plain}
\newtheorem{theorem}{Theorem}[section]

\newtheorem{lem}[theorem]{Lemma}
 \newtheorem{lemma}[theorem]{Lemma}

\newtheorem{remark}{Remark}[section]

\newcommand{\cx}{{\mathcal X}}

\newcommand{\mR}{{\mathbb R}}

  \renewcommand{\theequation}{\arabic{section}.\arabic{equation}}
  \makeatletter
  \@addtoreset{equation}{section}
  \makeatother

 \def\beqlb{\begin{eqnarray}}\def\eeqlb{\end{eqnarray}}
 \def\beqnn{\begin{eqnarray*}}\def\eeqnn{\end{eqnarray*}}

 \def\ar{\!\!&}

 \def\mbb{\mathbb}

 \def\qed{\hfill$\Box$\medskip}

\newcommand{\bcen}{\begin{center}}
\newcommand{\ecen}{\end{center}}
\newcommand{\bgeqn}{\begin{equation}}
\newcommand{\edeqn}{\end{equation}}

\def\dz{\delta}
\def\lz{\lambda}
\def\LZ{\Lambda}

\def\ez{\varepsilon}

\def\J{{\mathcal  J}}

\def\E{{\mbb  E}}

\def\mR{{\mbb  R}}

\renewcommand{\P}{\mathbb{P}}
\def\rar{\rightarrow}

\def\l{\left}
\def\r{\right}

 \def\ar{\!\!\!&}

\begin{document}

\title{On large deviation probabilities for empirical distribution of branching random walks: Schr\"{o}der case and B\"{o}ttcher case }

\author{Xinxin Chen and Hui He}

\maketitle

\noindent\textit{Abstract:}
 Given a super-critical branching random walk on $\mR$ started from the origin, let $Z_n(\cdot)$ be the counting measure which counts the number of individuals at the $n$-th generation located in a given set. Under some mild conditions, it is known in \cite{B90} that for any interval $A\subset \mR$, $\frac{Z_n(\sqrt{n}A)}{Z_n(\mR)}$ converges a.s. to $\nu(A)$,  where $\nu$ is the standard Gaussian measure.
 In this work, we investigate the convergence rates of $$\P\left(\frac{Z_n(\sqrt{n}A)}{Z_n(\mR)}-\nu(A)>\Delta\right),$$
 for $\Delta\in (0, 1-\nu(A))$, in both Schr\"{o}der case and B\"{o}ttcher case.

\bigskip


\bigskip


\section{Introduction and Main results}\label{Sec01}

\subsection{Branching random walk and its empirical distribution}\label{sec:empiric-dist}

Branching random walk, as an interesting object in probability theory, has been widely considered and well studied these years. The recent developments can be referred to Hu-Shi \cite{HuShi09}, Shi \cite{S12}. On the other hand, this object is closely related to many other random models, for example, multiplicative cascades, random fractals and random walk in random environment (see e.g. Hu-Shi \cite{HS07, HS15}, Liu \cite{Liu98, Liu06}).

Generally speaking, in a branching random walk, the point processes independently produced by the particles, formulated by the number of children and their displacements, all follow the so-called reproduction law. However, in this work, we consider a simpler model by assuming that the number of children and the motions are independent. Let us construct
it in the following way.

We take a Galton-Watson tree ${\cal T}$, rooted at $\rho$, with offspring distribution $\{p_k; k\geq0\}$. For any $u,v\in\mathcal{T}$, we write $u\preceq v$ if $u$ is an ancestor of $v$ or $u=v$. Moreover, to each node $v\in\mathcal{T}\setminus\{\rho\}$, we attach a real-valued random variable $X_v$ and define its position by
\[
S_v:=\sum_{\rho\prec u\preceq v}X_u.
\]
Let $S_\rho:=0$ for convenience. Suppose that given the tree $\mathcal{T}$, $\{X_v; v\in\mathcal{T}\setminus\{\rho\}\}$ are i.i.d. copies of some random variable $X$ (which is called step size). Thus, $\{S_u; u\in\mathcal{T}\}$ is a branching random walk. For any $n\in\mathbb{N}$, we introduce the following counting measure
\[
Z_n(\cdot)=\sum_{|v|=n, v\in {\cal T}}1_{\{S_v\in \cdot\}},
\]
where $|v|$ denotes the generation of node $v$, i.e., the graph distance between $v$ and $\rho$. Apparently, $Z_0=\delta_0$.

In this article, we always suppose that $p_0=0,\, p_1<1$. Note immediately that $\{Z_n(\mR); n\geq0\}$ is a supercritical Galton-Watson process with mean $ m :=\sum_{k\geq1} kp_k>1$. Assume that $m<\infty$. It is well known that the martingale
\[
W_n:=\frac{Z_n(\mR)}{m^n}
\]
 converges almost surely to some non-degenerate limit $W$  if and only if $\sum k\log k p_k<\infty$ (see e.g. \cite{AN72}).
Naturally, a central limit theorem on $Z_n(\cdot)$, was conjectured by Harris \cite{Ha63} and was proved by Asmussen-Kaplan \cite{AK76}, then was extended by Klebaner \cite{K82} and Biggins \cite{B90}. It says that if the step size $X$ has zero mean and  finite variance $\sigma^2$, for $A\in \mathcal{A}_0:=\{(-\infty,x]; x\in\mR\}$,
 \begin{eqnarray}\label{martlim}
\lim_{n\rar\infty} \frac{Z_n(\sqrt{n\sigma^2}A)}{m^n}=\nu(A)W,\quad a.s.
\end{eqnarray}
where $\nu$ is the  standard Gaussian measure on $\mR$. Let $\bar{Z}_n(\cdot):=\frac{Z_n(\cdot)}{Z_n(\mR)}$ be the corresponding empirical distribution. Then,
\begin{eqnarray}\label{emplim}
\lim_{n\rightarrow\infty} \bar{Z}_n(\sqrt{n\sigma^2}A)=\nu(A),\quad a.s.
\end{eqnarray}
What interests us is the convergence rates of \eqref{martlim} and \eqref{emplim}.

In the literature,  Asmussen and Kaplan \cite{AK76} proved that if $x_n=(\sigma\sqrt{n})x+o(\sqrt{n})$ with $x\in \mbb R$, then
$$
\lim_{n\rar\infty}\E\l(\l(m^{-n}Z_n((-\infty, x_n])-\nu((-\infty, x])W\r)^2\r)=0.
$$Recently, Chen \cite{C01}, under some regular assumptions, proved that  as $n\rar\infty$,
\begin{equation}\label{CGL}
 \frac{Z_n(\sqrt{n\sigma^2}A)}{m^n}-\nu(A)W=\frac{1}{\sqrt{n}}\xi(A)+o(\frac{1}{\sqrt{n}}),\textrm{ a.s.}
\end{equation}
where $\xi(A)$ is some explicitly defined  random variable.
Later, Gao and Liu \cite{GL16} generalized this convergence for a branching random walk in some random environment. They also, in \cite{GL17}, obtained the second and third orders in this asymptotic expansion. In particular, if $A=\mbb R$, one can refer to \cite{He71} and \cite{A94} for convergence rate of $W_n$ to $W$. Furthermore, by taking $A$ a singleton, the local version of this convergence has been investigated by R\'ev\'esz \cite{Re94}, Chen \cite{C01} and Gao \cite{G17}, even in higher dimensions.

In this paper, we aim at understanding the so-called large deviation behaviour of the convergence \eqref{emplim}, by considering the decaying rate of the following probability
\begin{eqnarray}\label{devA}
\P\left(\bar{Z}_n(\sqrt{n\sigma^2}A)-\nu(A)\geq \Delta\right),
\end{eqnarray}
with $\Delta>0$ a small constant. 

In fact, this problem has been investigated by Louidor and Perkins \cite{LP15} by assuming that $p_0=p_1=0$ (B\"{o}ttcher case) and that $X$ is simple random walk's step.  Recently, Louidor and Tsairi \cite{LT17} extend this result to the B\"{o}ttcher case with bounded step size by allowing dependence between the motions of children and their numbers.  However, if $X$ is not bounded, its tail distribution will be involved in the arguments and many regimes will appear in the asymptotic behaviours of \eqref{devA}.

In what follows, we will consider \eqref{devA} in the B\"{o}ttcher case by assuming that the step size $X$ has either Weibull tail distribution or Gumbel tail distribution. Also,
we will study this problem in the Sch\"{o}der case where $p_1>0=p_0$.

\subsection{Main results}
In the following of this paper, we always assume that ,
\begin{equation}\label{HypN}
 \E(e^{\theta Z_1(\mR)})=\sum_{k\geq0}p_ke^{\theta k}<\infty, \quad \text{for some }\theta>0,
\end{equation}
and that
\begin{equation}\label{HypX}
X \text{ is symmetric},\,\text{ with } E[X^2]=\sigma^2=1.
\end{equation}
\begin{remark}
The assumption of symmetry of $X$ is not necessary, but simplifies the proof. As long as $\Lambda(\cdot)$, defined in (\ref{LZI}) below,  is finite on an open interval including $0$, our arguments work.
The exact value of $E[X^2]$ does not play crucial role in our arguments either, but  simplifies notion. We only need that $X$ satisfies the classic central limit theorem.
\end{remark}
Before stating our main results, we first introduce some notations.
Let $\mathcal A$ be the algebra generated by
$\{(-\infty, x], x\in {\mbb R}\}.$  For $p\in(\nu(A), 1)$ with $A\in {\mathcal A}\setminus\emptyset$ such that $\nu(A)>0$, define
\beqlb\label{defI}
I_A(p)\ar=\ar\inf\l\{|x|: \nu(A-x)\geq p\r\};\\
\label{defJ}
J_A(p)\ar=\ar\inf\l\{y: \sup_{x\in {\mbb R}}\nu((A-x)/\sqrt{1-y})\geq p,\, y\in[0,1]\r\}.
\eeqlb
We say that $X$ fulfills Cram\'er's condition if  $E[e^{\kappa X}]<\infty, \text{for some }\, \kappa>0.$ In this case,
we can define the so-called logarithmic moment generating function and its inverse: for any $t\in\mR$ and $s\in\mR_+$,
\begin{equation}\label{LZI}
\Lambda(t)=:\log E[e^{tX}]\in [0, \infty]\quad\text{ and }\quad\Lambda^{-1}(s):=\inf\{t>0: \Lambda(t)\geq s\},
\end{equation}
 with the convention that $\inf\emptyset=\sup\{t>0: \Lambda(t)<\infty\}.$

\begin{theorem}[Schr\"{o}der case]\label{Main01}
Assume that $p_1>p_0=0$ and $X$ fulfills Cram\'er's condition. For $p\in (\nu(A),1)$, if $I_A(p)<\infty$ and $I_A(p)$ is continuous at $p$, then
\[
\lim_{n\rightarrow\infty}\frac{1}{\sqrt{n}}\log\P\left(\bar{Z}_n(\sqrt{n}A)\geq p\right)=-\Lambda^{-1}(\log\frac{1}{p_1})I_A(p).
\]
\end{theorem}
\begin{remark}
If we replace Cram\'er's condition in above theorem by  $\P(X>z)=\Theta(1) e^{-\lambda z^{\alpha}}$ as $z\rightarrow\infty$ with $\lambda>0$ and $0<\alpha<1$, then using same idea of proving (\ref{azle1}) below,  we may have
\[
\lim_{n\rightarrow\infty}\frac{1}{n^{\alpha/2}}\log\P\left(\bar{Z}_n(\sqrt{n}A)\geq p\right)=-\lambda I_A(p)^{\alpha}.
\]
\end{remark}

\begin{theorem}[Schr\"{o}der case]\label{Main02}
Assume that $p_1>p_0=0$. For $p\in (\nu(A),1)$, if $I_A(p)=\infty$ and $J_A(p)$ is continuous at $p$, then
\[
\lim_{n\rightarrow\infty}\frac{1}{{n}}\log\P\left(\bar{Z}_n(\sqrt{n}A)\geq p\right)= (\log {p_1})J_A(p).
\]
\end{theorem}
The next theorems concern the B\"{o}ttcher case where $p_0=p_1=0$. If $I_A(p+)=I_A(p)<\infty$, the decaying rate of (\ref{devA}) depends on the tail distribution of $X$, and also on the tree structure. In the following two theorems, we study two typical tail distributions of step size $X$: Weibull tail and Gumbel tail. Besides, we introduce
\[
b:=\min\{k\geq1: p_k>0\}\textrm{ and }B:=\sup\{k\geq 1: p_k>0\}\in [b, \infty].
\]
\begin{theorem}[B\"{o}ttcher case, Weibull tail]\label{Main03}
Assume that $p_0=p_1=0$. Suppose $\P(X>z)=\Theta(1) e^{-\lambda z^{\alpha}}$ as $z\rightarrow\infty$ for some constant $\alpha>0$ and $\lambda>0$. Take $p\in(\nu(A),1)$ such that $I_A(p)<\infty$ and $I_A(p)$ is continuous at $p$.
\begin{enumerate}
\item  If $\alpha\leq 1$ and $B>b$, then
\begin{equation}\label{azle1}
\lim_{n\rightarrow\infty}\frac{1}{n^{\alpha/2}}\log\P\left(\bar{Z}_n(\sqrt{n}A)\geq p\right) = -\lambda I_A(p)^{\alpha}.
\end{equation}
\item If $\alpha>1$ and $B>b$, then
\begin{equation}
\lim_{n\rightarrow\infty}\frac{(\log n)^{\alpha-1}}{n^{\alpha/2}}\log\P\left(\bar{Z}_n(\sqrt{n}A)\geq p\right) = -
\left(\frac{2\log b\log B}{\alpha(\log B-\log b)}\right)^{\alpha-1} \lambda I_A(p)^{\alpha}.
\end{equation}
\item If $B=b$, then there exist $C_\alpha>c_\alpha>0$ such that
\begin{equation}
-C_\alpha\leq\liminf_{n\rightarrow\infty}\frac{1}{n^{\alpha/2}}\log\P\left(\bar{Z}_n(\sqrt{n}A)\geq p\right)
\leq \limsup_{n\rightarrow\infty}\frac{1}{n^{\alpha/2}}\log\P\left(\bar{Z}_n(\sqrt{n}A)\geq p\right) \leq -c_\alpha.
\end{equation}
\end{enumerate}
\end{theorem}

\begin{theorem}[B\"{o}ttcher case, Gumbel tail]\label{Main04}
Assume that $p_0=p_1=0$. Suppose $\P(X>z)=\Theta(1)e^{-e^{ z^{\alpha}}}$ as $z\rightarrow\infty$ for some constant $\alpha>0$. For $p\in(\nu(A),1)$ such that $I_A(p)<\infty$ and $I_A(p)$ is continuous at $p$, we have
\begin{eqnarray}\label{Main0401}
\lim_{n\rightarrow\infty}{n^{-\frac{\alpha}{2(\alpha+1)}}}\log\left[-\log
\P\left(\bar{Z}_n(\sqrt{n}A)\geq p\right)\right]
= \left(y_{\alpha}I_A(p)\log b\right)^{\frac{\alpha}{\alpha+1}},
\end{eqnarray}
where $y_\alpha:=\frac{(1+\alpha)\log B}{(1+\alpha)\log B-\log b} $.
\end{theorem}

\begin{remark}
To obtain the exact decaying rates, we take advantage of the randomness of the embedding tree in the arguments. However, if we consider a regular tree ($B=b$) and motions with Weibull tail, such idea does not work any more and the situation becomes more delicate.  We believe that in this case, there is a close link between the decaying rate and the a.s. convergence \eqref{CGL}.
\end{remark}

To accomplish this work, we also state the result when $X$ is bounded, which is obtained by Louidor and Tsairi in \cite{LT17}.

\begin{theorem}[B\"{o}ttcher case, Theorem 1.2 of \cite{LT17}]\label{Main05}
Assume that $p_0=p_1=0$. If $\text{ess sup }X=L$ for some $0<L<\infty$,  for $p\in(\nu(A),1)$ such that $I_A(p)<\infty$ and $I_A(p)$ is continuous at $p$, we have
\begin{equation}\label{LPThe}
\lim_{n\rightarrow\infty}\frac{1}{n^{1/2}}\log\left[-\log\P\left(\bar{Z}_n(\sqrt{n}A)\geq p\right) \right]= \frac{I_A(p)\log b}{L}.
\end{equation}
\end{theorem}

The following result is universal, regardless of the tail distribution of $X$, when $I_A(p)=\infty$.

\begin{theorem}[B\"{o}ttcher case, Theorem 1.2 of \cite{LT17}]
Assume that $p_0=p_1=0$.  For $p\in(\nu(A),1)$ such that $I_A(p)=\infty$ and $J_A(p)$ is continuous at $p$, we have
\[
\lim_{n\rightarrow\infty}\frac{1}{{n}}\log\left[-\log\P\left(\bar{Z}_n(\sqrt{n}A)\geq p\right) \right]=J_A(p)\log b.
\]
\end{theorem}


At the end of this section, let us say a couple of words on strategy of proofs, which is partially inspired by that of Loudior and Perkins. To have $\{\bar{Z}_n(\sqrt{n}A)\geq p\}$, we take an intermediate generation $t_n$ and suppose that most individuals at this generation are positioned around $x\sqrt{n}$, so that finally $\bar{Z}_n(\sqrt{n}A)\approx \nu_{n-t_n}(\sqrt{n}(A-x))\approx \nu(\frac{n}{n-t_n}(A-x))\gtrsim p$. This brings out the definitions of $I_A(p)$ and $J_A(p)$. If $I_A(p)<\infty$, we take $t_n=o(n)$; otherwise, we take $t_n=\Theta(n)$. Moreover, the effort made up to generation $t_n$ depends not only on branching but also on motions, which brings out the different treatments in different regimes.

The rest of this paper is organised as follows. In Sect. 2, we present some
basic facts on random walks and branching random walks which will be used frequently in our proofs of main results. We study the Schr\"oder case in Sect. 3, where Theorems \ref{Main01} and \ref{Main02} are proved. In Sect. 4,  B\"ottcher case is treated.  Theorems \ref{Main03} and \ref{Main04} will be  proved. Let $C_1, C_2, \cdots$ and $c_1, c_2, \cdots$ denote positive constants which might change from line to line. As usual, $f_n=O(g_n)$ or $f_n=O(1)g_n$ mean that $f_n\leq Cg_n$ for some $C>0$ and all $n\geq1$. $f_n=\Theta(1)g_n$ means that $f_n$ is bounded both above and below by $g_n$ asymptotically. $f_n=o(g_n)$ or $f_n=o_n(1)g_n$ mean that $\lim_{n\rightarrow \infty}\frac{f_n}{g_n}=0$.
\section{Preliminary results}

In this section, we present some well-known facts and useful lemmas, which will be applied frequently in the next sections.  Denote by $\nu_n:=\underbrace{\P_X*\cdots *\P_X}_{n\textrm{ times }}$, the distribution of a $X$-random walk at the $n$-th step. Recall that $\nu$ represents the standard normal distribution on the real line. The following lemma states some basic facts about $\nu$ and $\nu_n$.

\begin{lemma}\label{LP}
Let $A\in{\mathcal A}\setminus\emptyset$ and $p\in (0, 1)$.
\begin{enumerate}
\item The mapping $(a, b) \mapsto \nu(aA+b)\in C^{\infty}({\mbb R}^2)$. Moreover, it is Lipschitz on $b$ and is uniformly continuous on $K\times\mR$ for any compact set $K$.
\item If $0<I_A(p)<\infty$, then there exists $x\in \mbb R$ with $|x|=I_A(p)$ such that $\nu(A-x)\geq p.$
\item  $0\leq J_A(p)<1$ and  there exists $x\in \mbb R$  and $r\in (0, 1)$ with $r=J_A(p)$ such that $\nu((A-x)/\sqrt{1-r})\geq p.$
\item If $p>\nu(A)$, then either $0<I_A(p)<\infty$ or $I_A(p)=\infty$, $J_A(p)\in (0,1).$
\item \begin{description}
\item[ (i)] Let $A^-_\ez:=\{x\in A: B(x,\ez)\subset A\}$ be the $\ez$-interior of $A$, and $A^+_\ez:=\cup_{x\in A}B(x,\ez)$ be the $\ez$-neighbourhood of $A$. Then
\[
A^-_\ez\subset\cap_{y\in[-\ez/2,\ez/2]}(A-y) \textrm{ and } A\setminus A_\ez^-\subset (\partial A)^+_{\ez}
\]
and
\[
\cup_{y\in[-\ez/2,\ez/2]}(A-y)\subset A_{\ez}^+\textrm{ and } A_{\ez}^+\setminus A\subset (\partial A)^+_{\ez},
\]
where $(\partial A)^+_{\ez}$ is the $\ez$-neighbourhood of $\partial A$ and $\lim_{\ez\downarrow 0}\nu((\partial A)^+_{\ez})=\nu(\partial A)=0$.
\item[(ii)] Moreover,
\[
\sup_{x\in \mR}|\nu(A_{\ez}^+-x)-\nu(A-x)|=\sup_{x\in\mR}|\nu(A_{\ez}^+\setminus A-x)|\leq \sup_x\nu((\partial A-x)^+_{\ez}).
\]
If $Leb(\cdot)$ is the Lebesgue measure on $\mR$, we have $\text{Leb}(\partial A)=\nu(\partial A)=0$, and
\[
\sup_x\nu((\partial A-x)^+_{\ez})\leq Leb((\partial A)^+_\ez)\rightarrow 0, \textrm{ as } \ez\downarrow 0.
\]
\end{description}
\item As $\nu(\partial A)=0$, for any $l>1$, we have the following uniform convergence,
$$
\lim_{n\rightarrow\infty}\sup_{a\in [l^{-1}, l]}\sup_{b\in {\mbb R}}\left|\nu_n(\sqrt{n}(aA+b))-\nu(aA+b)\right|=0.
$$
\end{enumerate}
\end{lemma}

In fact, (1)-(4) and (6) can be found in Lemmas 2.1, 2.2 and 2.4 of \cite{LP15}, (5) is a basic property. So we feel free to omit its proof.

Let $\mathcal{M}$ be the collection of locally finite counting measures on $\mathbb{R}$. For any $\zeta\in\mathcal{M}$ which is finite, we can write it as
\[
\zeta=\sum_{i=1}^{|\zeta|}\delta_{x_i}
\]
with $x_i\in\mathbb{R}$ and $|\zeta|<\infty$ the total mass. For convenience, we write $x\in\zeta$ if $\zeta\{x\}\geq1$. Let $\{Z_n^\zeta\}$ be the branching random walk started from $Z_0^\zeta=\zeta$. Similarly, let $\bar{Z}_n^\zeta(\cdot)$ be the corresponding empirical distribution. Because of \eqref{HypN}, we have the following lemma, borrowed from \cite{LP15}.

\begin{lemma}\label{LP+}
There exists $C_1, C_2>0$ such that for all $\Delta>0$ sufficiently small and $n\geq1$, for any finite $\zeta\in\mathcal{M}$,
    \beqlb\label{lem2.1LP01}
    \P\left(\bar{Z}_n^{\zeta}(A)>\frac{1}{|\zeta|}\sum_{x\in \zeta}\nu_n(A-x)+\Delta\right)\leq C_{1}e^{-C_{2} \Delta^2|\zeta|},
    \eeqlb
   The same holds if $>, +\Delta$ are replaced by $<, -\Delta$, respectively.
\end{lemma}

The next two lemmas are slightly stronger versions of (\ref{emplim}).
\begin{lemma}\label{lemuniform}
Let $A\in{\mathcal A}\setminus\emptyset$. Let $\{a_n: n\geq1\}$ and $\{b_n: n\geq1\}$ be two deterministic sequences such that $a_n\rightarrow1$ and $b_n\rightarrow0$. Then as ${n\rightarrow\infty}$,
\beqlb\label{lemuniform01}
\bar{Z}_n(\sqrt{n}(a_nA+b_n))\overset{a.s.}{\longrightarrow}\nu(A).
\eeqlb
\end{lemma}
We feel free to omit its proof as it is a direct consequence of \eqref{emplim}.

\begin{lemma}\label{lemloc}
Take the same assumptions as in Lemma \ref{lemuniform}. Then for any finite $\zeta\in \mathcal{M}$, as ${n\rightarrow\infty}$,
\beqlb\label{lemlocal01}
\bar{Z}_n^{\sqrt{n}\zeta}(\sqrt{n}(a_nA+b_n))\overset{a.s.}{\longrightarrow}\frac{\sum_{x\in \zeta}\nu(A-x)W^x}{\sum_{x\in\zeta}W^x},
\eeqlb
where $\sqrt{n}\zeta=\sum_{x\in\zeta}\delta_{x\sqrt{n}}$ and $\{W^x: x\in\zeta\}$ are i.i.d. random variables distributed as $W$.
 \end{lemma}
  \begin{proof} For any $A\in {\mathcal A}$ and $x\in\mR$, as $n\rightarrow \infty$,
 $$
\frac{Z_n^{\delta_{x\sqrt{n}}}(\sqrt{n}A)}{m^n}-W\nu(A-x)
\overset{a.s.}{\longrightarrow}0.
 $$
 This convergence, together with the  branching property of Galton-Watson process, gives
 \beqlb\label{lemlocal02}
\bar{Z}_n^{\sqrt{n}\zeta}(\sqrt{n}A)\overset{a.s.}{\longrightarrow}\frac{\sum_{x\in \zeta}\nu(A-x)W^x}{\sum_{x\in\zeta}W^x}.
\eeqlb
We thus conclude by Lemma \ref{lemuniform}.
 \end{proof}

When $p_1>0$, denote by $\chi$ the so-called Schr\"oder constant  with
\beqlb\label{defsch}
m^{-\chi}=p_1.
\eeqlb
We also recall here a result from Lemma 13 in \cite{FW08}. Note that $Z_n(\mR)=|Z_n|$ for the counting measure $Z_n$.
\begin{lemma}\label{lemFW08}
If $p_1>0$, then there exists a constant $C_3>0$ such that
\beqlb\label{local}
P(|Z_n|=k)\leq C_3\l(\frac{1}{k}\wedge\left( k^{\chi-1}p_1^n\right)\r),\quad \forall k,\, n\geq1.
\eeqlb
\end{lemma}
The next lemma is the well-known Cram\'er theorem; see Theorem 3.7.4 in \cite{DZ98}. Recall that $\nu_n$ the distribution of a $X$-random walk at $n$-th step.
\begin{lemma}
If $\E[e^{\kappa X}]<\infty$ for some $\kappa>0$, for any  $a>0$ and $\ez>0$, as $n\rightarrow\infty$,
\beqlb\label{LDPRW}
\lim_{n\rightarrow\infty}\frac{1}{n}\log\nu_n((-(a+\epsilon)n, -an])= -\gamma(a):=-\sup_{t>0}\{at-\Lambda(t)\}.
\eeqlb
\end{lemma}
For $a<0$, we have a similar result with $\gamma(a)=\gamma(-a)$. This result follows immediately from Cram\'er's theorem, which says that for any $a>0$,
\[
\lim_{n\rightarrow\infty}\frac{1}{n}\log\nu_n((-\infty, -an])=-\gamma(a),
\]
where the rate function $\gamma(a)\geq 0$ and convex.

\section{Schr\"{o}der case}
\subsection{Proof of Theorem \ref{Main01}: $I_A(p)<\infty$ and $I_A$ is continuous at $p$}
In this section, we are going to demonstrate Theorem \ref{Main01}.

For the lower bound, we will make a single branch up to some generation $O(\sqrt{n})$, and the random walk along this branch moves to the level $x\sqrt{n}$ so that the descendants at the $n$-th generation behave like $\bar{Z}_n(\sqrt{n}A)\approx \nu(A-x)\geq p$ with high probability. For the upper bound, we will begin with a rough bound. Then an iteration method will be applied to improve the bound and to obtain the exact limit.
\subsubsection{Lower bound}
Recall  $\Lambda(t)=\log E[e^{tX}]$, (\ref{LZI}) and \eqref{LDPRW}.  Define
$$
\bar{\LZ}(p_1):=\inf_{a>0} \frac{\log \frac{1}{p_1}+\gamma(a)}{a}.
$$
One could check that
\begin{eqnarray}\label{lem01low03}
\bar{\LZ}(p_1)=\LZ^{-1}(\log\frac{1}{p_1})\in(0,\infty).
\end{eqnarray}
\begin{lemma}\label{lem01low} ({\bf Lower bound}) If all assumptions in Theorem \ref{Main01} hold, then for any $\ez>0$,
\begin{eqnarray}\label{lem01low01}
\liminf_{n\rightarrow\infty}\frac{1}{\sqrt{n}}\log {\mbb P}\left(\bar{Z}_n(\sqrt{n}A)\geq p\right)\geq -(I_A(p)+\ez)\LZ^{-1}(\log\frac{1}{ p_1}).
\end{eqnarray}
\end{lemma}
\begin{proof} Since $I_A(p)$ is continuous at $p$, by (2) of Lemma \ref{LP}, there exist $x\in {\mbb R}, \delta>0$ such that
$$
\nu(A-x)\geq p+\delta\quad\text{and}\quad |x|=I_A(p)+\ez.
$$
We only consider the case where $x<0$. The case where $x>0$ can be treated similarly. For any $a>0$, let $C_a:=\frac{I_A(p)+\ez }{a} $ and $t_n:= \lfloor C_a\sqrt{n}\rfloor$. Note that $
-aC_a=x$. Observe that for any $\eta>0$,
\[
\P(\bar{Z}_n(\sqrt{n}A)\geq p)\geq \P\left(\bar{Z}_n(\sqrt{n}A)\geq p, Z_{t_n}(\mR)= Z_{t_n}((-a(1+\eta) t_n, -at_n ])=1\right)
\]
which by Markov property at time $t_n$ implies that,
\begin{align}
&\P(\bar{Z}_n(\sqrt{n}A)\geq p)\nonumber\\
&\geq \P\left( Z_{t_n}(\mR)=Z_{t_n}((-a(1+\eta) t_n, -at_n ])=1\right)\inf_{y\in (-a(1+\eta) t_n, -at_n ]}\P\left(\bar{Z}_{n-t_n}^{\delta_y}( \sqrt{n}A)\geq p\right)\nonumber\\
&= p_1^{t_n}\nu_{t_n}(-(a(1+\eta) t_n, -at_n ] )\inf_{y\in ((1+\eta)x\sqrt{n}, x\sqrt{n} ]}\P\left(\bar{Z}_{n-t_n}( \sqrt{n}A-y)\geq p\right).\label{1tn}
\end{align}
It remains to treat the probability on the right hand side. Let $A(x, \eta):=\cap_{y_0\in ((1+\eta)x, x]}(A-y_0)$. Observe that for any $y\in(-a(1+\eta)C_a\sqrt{n}, -aC_a\sqrt{n} ]$,
\[
\bar{Z}_{n-t_n}( \sqrt{n}A-y)\geq \bar{Z}_{n-t_n}(\sqrt{n} A(x,\eta))=\bar{Z}_{n-t_n}\left( \sqrt{n-t_n}\left(\frac{\sqrt{n}}{\sqrt{n-t_n}}A(x,\eta)\right)\right).
\]
On the other hand, by (5) of Lemma \ref{LP}, for $\eta>0$ small enough,
\[
\nu(A(x,\eta))\geq \nu(A-x)-\delta/2\geq p+\delta/2.
\]
It follows that
\begin{align*}
&\inf_{y\in ((1+\eta)x\sqrt{n}, x\sqrt{n} ]}\P\left(\bar{Z}_{n-t_n}( \sqrt{n}A-y)\geq p\right)\\ &\qquad \geq  \P\left(\bar{Z}_{n-t_n}\left( \sqrt{n-t_n}\left(\frac{\sqrt{n}}{\sqrt{n-t_n}}A(x,\eta)\right)\right)\geq p\right)\\
&\qquad\geq  \P\left(\bar{Z}_{n-t_n}\left( \sqrt{n-t_n}\left(\frac{\sqrt{n}}{\sqrt{n-t_n}}A(x,\eta)\right)\right)\geq \nu(A(x,\eta))-\delta/2\right),
\end{align*}
which, as $n\rightarrow\infty$, converges to $1$ in view of Lemma \ref{lemuniform}. Going back to \eqref{1tn}, this means that for $\eta>0$ small enough and $n$ large enough,
\[
\P(\bar{Z}_n(\sqrt{n}A)\geq p)\geq \frac{1}{2}p_1^{t_n}\nu_{t_n}\Big((-a(1+\eta) t_n, -at_n ] \Big).
\]
This yields that
\[
\liminf_{n\rightarrow\infty}\frac{1}{\sqrt{n}}\log {\mbb P}\left(\bar{Z}_n(\sqrt{n}A)\geq p\right)\geq C_a\log p_1+\liminf_{n\rightarrow\infty} \frac{C_a}{t_n}\log \nu_{t_n}\Big((-a(1+\eta) t_n, -at_n ] \Big),
\]
which, by \eqref{LDPRW}, implies
\[
\liminf_{n\rightarrow\infty}\frac{1}{\sqrt{n}}\log {\mbb P}\left(\bar{Z}_n(\sqrt{n}A)\geq p\right)\geq -C_a\left(\log\frac{1}{p_1}+\gamma(a)\right).
\]
As the limit on the right hand side does not depend on $a$, we obtain that
\[
\liminf_{n\rightarrow\infty}\frac{1}{\sqrt{n}}\log {\mbb P}\left(\bar{Z}_n(\sqrt{n}A)\geq p\right)\geq -(I_A(p)+\ez)\inf_{a>0}\frac{\log\frac{1}{p_1}+\gamma(a)}{a}=- (I_A(p)+\ez) \bar{\LZ}(p_1),
\]
which implies the lower bound. \end{proof}

\begin{remark}
If $X$ is a simple random walk, then we have
$$
\bar{\LZ}(p_1)=\log \l(1+\sqrt{1-p_1^2} \r)-\log p_1\in (0,\infty).
$$
\end{remark}

\subsubsection{Upper bound} In this subsection, we are going to show that
\begin{equation}\label{lsup}
\limsup_{n\rightarrow\infty}\frac{1}{\sqrt{n}}\log {\mbb P}\left(\bar{Z}_n(\sqrt{n}A)\geq p\right)\leq -I_A(p)\bar{\LZ}(p_1).
\end{equation}

We believe that the strategy of lower bound is somehow optimal. For the upper bound, we consider also some intermediate generation of order $O(\sqrt{n})$, where, to get $\{\bar{Z}_n(\sqrt{n}A)\geq p\}$, the population size should be atypically small and the extreme positions should be close to $\pm I_A(p)\sqrt{n}$.

For $\ez\in(0,1/2)$ small, $a\in\mR_+$, define
\[
N(a)=[a(1-\ez)\sqrt{n}]\quad
\text{and}  
\quad
B_n:=[-I_A(p)(1-\ez)\sqrt{n}, I_A(p)(1-\ez)\sqrt{n}].
\]
Moreover, there exists $\delta\in(0, p-\nu(A))$ small enough such that
\begin{equation}\label{nvB}
\sup_{z\in B_n/\sqrt{n}}\nu(A-z)\leq p-\delta.
\end{equation}
The following lemma states the idea presented at the beginning of this section. It gives also a rough upper bound. Let $\mathcal{M}_B$ be the collection of all locally finite counting measures on $\mR$ which vanish outside $B$; i.e., ${\mathcal M}_B=\{\zeta\in{\mathcal M}: \zeta(B^c)=0\}$.  Recall that $|\zeta|$ is the total mass of $\zeta\in\mathcal{M}$.
\begin{lem}\label{rup}
Suppose that all assumptions in Theorem \ref{Main01} hold. Then there exists $\alpha>0$, such that for $a=I_A(p)\alpha$ and $n$ large enough,
\begin{equation}\label{Bc}
\P(Z_{N(a)}(B_n^c)\geq 1) =o_n(1)\P(\bar{Z}_n(\sqrt{n}A)\geq p)
\end{equation}
Moreover, we have
\begin{equation}\label{mainevent}
\P(\bar{Z}_n(\sqrt{n}A)\geq p) =(1+o_n(1))\P(Z_{N(a)}\in\mathcal{M}_{B_n}, |Z_{N(a)}|\leq n, \bar{Z}_n(\sqrt{n}A)\geq p).
\end{equation}
Consequently, there exists some constant $C_4>0$ such that for all $n\geq1$,
\begin{equation}\label{rub}
\P(\bar{Z}_n(\sqrt{n}A)\geq p) \leq C_4 p_1^{\alpha I_A(p)(1-\ez)\sqrt{n}}
\end{equation}
\end{lem}


\begin{proof} The proof will be divided into two subparts.\\
{\it Subpart 1:} We shall prove (\ref{Bc}). Observe that by symmetry of $\nu_N$ and Markov inequality,
\begin{align*}
\P(Z_{N(a)}(B_n^c)\geq 1) &\leq 2\P\left(Z_{N(a)}\Big([I_A(p)(1-\ez)\sqrt{n},+\infty)\Big)\geq 1\right)\\
&\leq 2\E\left[Z_{N(a)}\Big([I_A(p)(1-\ez)\sqrt{n},+\infty)\Big)\right]\\
&=2m^{N(a)}\nu_{N(a)}\Big([I_A(p)(1-\ez)\sqrt{n},+\infty)\Big).
\end{align*}
By Chernoff bound of Cram\'er's theorem, one sees that
\[
\nu_{N(a)}\Big([I_A(p)(1-\ez)\sqrt{n},+\infty)\Big)
=\nu_{N(a)}\Big(\left[\frac{I_A(p)}{a}{N(a)},+\infty\right)\Big) \leq \exp\left\{-{N(a)}\gamma\left(\frac{I_A(p)}{a}\right)\right\},
\]
which implies that
\begin{align}
\P(Z_{N(a)}(B_n^c)\geq 1)\leq &2e^{-{N(a)}(\gamma(\frac{I_A(p)}{a})-\log m)}\nonumber\\
=& 2 e^{-\sqrt{n} (1-\ez)(a\gamma(\frac{I_A(p)}{a})-a\log m)}.
\end{align}
Meanwhile,
\beqnn
\lim_{a\downarrow 0}a(\gamma(I_A(p)/a)-\log m)\ar=\ar I_A(p)\lim_{a\rightarrow+\infty } \frac{\gamma(a)-\log m}{a}\cr\ar=\ar I_A(p) \lim_{a\rightarrow+\infty }\frac{\gamma(a)+\log \frac{1}{p_1}}{a}>I_A(p)\bar{\LZ}(p_1).
\eeqnn
where the inequality follows from the definition of  $\bar{\LZ}(p_1)$.
Therefore,  for any $\alpha\in(0,1)$ small enough and any $\ez>0$ small enough, with $a=I_A(p)\alpha$, we have
\[
\limsup_n\frac{1}{\sqrt{n}}\log \P(Z_{N(a)}(B_n^c)\geq1)\leq - (1+\ez)I_A(p)\bar{\LZ}(p_1).
\]
Therefore, in view of the lower bounded obtained in Lemma \ref{lem01low}, one could choose $\alpha>0$ and $\ez_0>0$ sufficiently small  such that  for $\ez\in(0,\ez_0)$,
\[
\P(Z_{N(a)}(B_n^c)\geq1)=o_n(1)\P(\bar{Z}_n(\sqrt{n}A)\geq p),
\]
which is just (\ref{Bc}).\\

{\it Subpart 2:}  We turn to prove (\ref{mainevent}) and (\ref{rub}).
According to (\ref{Bc}), to bound $\P(\bar{Z}_n(\sqrt{n}A)\geq p)$, we only need to consider $\P(Z_{N(a)}\in\mathcal{M}_{B_n}, \bar{Z}_n(\sqrt{n}A)\geq p)$. Note that by Markov property at time $N(a)$,
\begin{align}\label{upbdN}
\P(Z_{N(a)}\in\mathcal{M}_{B_n}, \bar{Z}_n(\sqrt{n}A)\geq p)=&\sum_{\zeta\in\mathcal{M}_{B_n}}\P(Z_{N(a)}=\zeta, \bar{Z}_n(\sqrt{n}A)\geq p)\nonumber\\
=&\sum_{\zeta\in\mathcal{M}_{B_n}}\P(Z_{N(a)}=\zeta)\P\left(\bar{Z}_{n-N(a)}^\zeta(\sqrt{n}A)\geq p\right).
\end{align}
By (1) and (6) of Lemma \ref{LP}, we have
\[
\sup_{z\in B_n}|\nu_{n-{N(a)}}(\sqrt{n}A-z)-\nu(A-\frac{z}{\sqrt{n}})|=o_n(1).
\]
So, by \eqref{nvB}, for $n$ large enough and any $\zeta\in \mathcal{M}_{B_n}$, one has
\[
p\geq \sup_{z\in B_n}\nu(A-\frac{z}{\sqrt{n}})+\delta\geq \frac{1}{|\zeta|}\sum_{z\in\zeta}\nu_{n-N(a)}(\sqrt{n}A-z)+\delta/2.
\]
Therefore, \eqref{upbdN} becomes
\beqnn
\ar\ar\P(Z_{N(a)}\in\mathcal{M}_{B_n}, \bar{Z}_n(\sqrt{n}A)\geq p)\cr\ar\ar\qquad\leq \sum_{\zeta\in\mathcal{M}_{B_n}}\P(Z_{N(a)}=\zeta)\P\left(\bar{Z}^\zeta_{n-N(a)}(\sqrt{n}A)\geq \frac{1}{|\zeta|}\sum_{z\in\zeta}\nu_{n-N(a)}(\sqrt{n}A-z)+\delta/2\right),
\eeqnn
which, by Lemma \ref{LP+},  is bounded by $C_1\sum_{\zeta\in\mathcal{M}_{B_n}}\P(Z_{N(a)}=\zeta) e^{-C_2|\zeta|\delta^2/4}$. This gives that
\beqlb\label{upbdN01}
\P(Z_{N(a)}\in\mathcal{M}_{B_n}, \bar{Z}_n(\sqrt{n}A)\geq p)\leq C_1\sum_{k=1}^\infty \P(|Z_{N(a)}|=k)e^{-c_1 k}.
\eeqlb
As a consequence of \eqref{local},
\begin{eqnarray}\label{upbdN02}
\P(Z_{N(a)}\in\mathcal{M}_{B_n}, \bar{Z}_n(\sqrt{n}A)\geq p)\leq c_2 p_1^{N(a)}.
\end{eqnarray}
Furthermore, following the same arguments to get (\ref{upbdN01}), we have
\begin{align}\label{upbdN03}
\P(Z_{N(a)}\in\mathcal{M}_{B_n}, |Z_{N(a)}|\geq n, \bar{Z}_n(\sqrt{n}A)\geq p)& \leq C_1\sum_{k=n}^\infty \P(|Z_{N(a)}|=k)e^{-c_1 k}\nonumber\\\
& \leq C_1e^{-c_1 n}\nonumber\\
&=o_n(1)\P(\bar{Z}_n(\sqrt{n}A)\geq p),
\end{align}
where the last equality follows from Lemma \ref{lem01low}.
Then (\ref{Bc}) and (\ref{upbdN03}) yield that
\begin{eqnarray}
\ar\ar\P(\bar{Z}_n(\sqrt{n}A)\geq p)\cr\ar\ar\quad\leq
\P\left( Z_{N(a)}\notin\mathcal{M}_{B_n}\right)+\P\left(Z_{N(a)}\in\mathcal{M}_{B_n}, \bar{Z}_n(\sqrt{n}A)\geq p\right)
\cr\ar\ar\quad=
 o_n(1)\P\left(\bar{Z}_n(\sqrt{n}A)\geq p\right)+\P(Z_{N(a)}\in\mathcal{M}_{B_n}, |Z_{N(a)}|<n, \bar{Z}_n(\sqrt{n}A)\geq p),
\end{eqnarray}
which, together with (\ref{upbdN02}), gives (\ref{mainevent}) and (\ref{rub}).
\end{proof}
More generally, for $\rho_n=1+o_n(1)$, we also have
\[
\P(\bar{Z}_n(\sqrt{n}(\rho_nA))\geq p)\leq C_5 e^{-\log \frac{1}{p_1}I_A(p)\alpha(1-\ez)\sqrt{n}}.
\]
The arguments above still work by remarking (1) of Lemma \ref{LP} which says that
\[
\sup_{A\in\mathcal{A}}|\nu(\rho_n A)-\nu(A)|=O( |\rho_n-1|)=o_n(1)
\]
and that
\[
\sup_{A\in\mathcal{A}}|\nu(A-z)-\nu(A-y)|=O( |y-z|).
\]
In what follows, we fix $\alpha>0$, the real number taken in the previous lemma and take $\ez\in(0,\ez_0)$. Here is the key lemma for the iteration of the upper bound.
\begin{lem}\label{iteration}
Suppose that all assumptions in Theorem \ref{Main01} hold. Take $\rho_n=1+o_n(1)$.
If we have
\begin{equation}\label{upbdL}
\limsup_n\frac{\log\P(\bar{Z}_n(\sqrt{n}\rho_nA)\geq p)}{\sqrt{n}}\leq -LI_A(p),
\end{equation}
for some $L>0$, then
\begin{equation}\label{upbdL0}
\limsup_n\frac{\log\P(\bar{Z}_n(\sqrt{n}\rho_nA)\geq p)}{\sqrt{n}}\leq -F(L)I_A(p),
\end{equation}
where
\[
F(L):=\alpha\inf_{u\in\mR}\left(\log\frac{1}{p_1}+\gamma(u)-uL\right)+L
\]
with $\alpha>0$ chosen in Lemma \ref{rup}.
\end{lem}
\begin{remark}
By Lemma \ref{rup}, we could take $L=\alpha\log \frac{1}{p_1}$. One can see that $F(L)\leq \bar{\Lambda}(p_1)$ for $L\leq \bar{\Lambda}(p_1)$ in Lemma \ref{A1}.
\end{remark}
\begin{proof}Recall that $a=\alpha I_A(p)$ is fixed here. By \eqref{mainevent}, it suffices to consider $\P(Z_{N(a)}\in\mathcal{M}_{B_n}, |Z_{N(a)}|<n, \bar{Z}_n(\sqrt{n}\rho_nA)\geq p)$ with $N(a)=[a(1-\ez)\sqrt{n}]$. In the rest of this proof, we write $N$ and $B$ for $N(a)$  and $B_n$, respectively. Let $\rho'_n:=\rho_n\frac{\sqrt{n}}{\sqrt{n-N}}$.
 We  observe that  given $\{Z_N=\zeta\}$,
\begin{eqnarray}\label{iter02}
\bar{Z}_n(\sqrt{n}\rho_nA)=\bar{Z}^\zeta_{n-N}(\sqrt{n}\rho_nA)\leq \max_{z\in\zeta}\bar{Z}_{n-N}^z(\sqrt{n-N} (\rho'_nA)),
\end{eqnarray}
where the last inequality follows from the elementary inequality
\[
\frac{\sum_{i=1}^k a_i}{\sum_{i=1}^k b_i}\leq \max_{i\leq k}\frac{a_i}{b_i}, \quad\forall a_i\geq 0,\, b_i>0.
\]
Then (\ref{iter02}) gives
\begin{eqnarray}\label{iter01}
\ar\ar\P(Z_N\in\mathcal{M}_B, |Z_N|<n, \bar{Z}_n(\sqrt{n}\rho_nA)\geq p)\cr
\ar\ar\quad\leq \P\left(Z_N\in\mathcal{M}_B, |Z_N|<n, \max_{z\in Z_N}\bar{Z}_{n-N}^z(\sqrt{n-N} (\rho'_nA))\geq p\right).
\end{eqnarray}
For any $M>0$, let us consider a partition on $B$:
\[
(1-\ez)\sqrt{n}u_0<(1-\ez)\sqrt{n}u_1<\cdots<(1-\ez)\sqrt{n}u_M,
\]
where $u_{i+1}-u_i=\eta:= \frac{2I_A(p)}{M}$ with $u_0=-I_A(p)$ and $u_M=I_A(p)$.  Then the r.h.s of (\ref{iter01}) is less than
\[
\P\left(\bigcup_{i=0}^{M-1}\left\{Z_N\in\mathcal{M}_B, |Z_N|<n, \max_{z\in Z_N\cap J_i }\bar{Z}_{n-N}^z(\sqrt{n-N} (\rho'_nA))\geq p\right\}\right),
\]
where $J_0=\left[(1-\ez)\sqrt{n}u_0,\, (1-\ez)\sqrt{n}u_{1}\right]$ and $$ J_i=\left((1-\ez)\sqrt{n}u_i,\, (1-\ez)\sqrt{n}u_{i+1}\right],\quad 1\leq i\leq M-1.$$
 It follows from Markov property that
\begin{align}\label{upbdA}
&\P(Z_N\in\mathcal{M}_B, |Z_N|<n, \bar{Z}_n(\sqrt{n}\rho_nA)\geq p)\cr
\leq &\sum_{i= 0}^{M-1}\sum_{\zeta\in\mathcal{M}_{B}, |\zeta|\leq n}\P\left(Z_N=\zeta \right)\P\left(\max_{\substack{z\in \zeta\cap J_i}}\bar{Z}_{n-N}^z(\sqrt{n-N} (\rho'_nA))\geq p\right)\cr
\leq &\sum_{i= 0}^{M-1}\sum_{\zeta\in\mathcal{M}_{B}, |\zeta|\leq n}\P\left(Z_N=\zeta \right)\sum_{\substack{z\in \zeta\cap J_i}}\P\left(\bar{Z}_{n-N}(\sqrt{n-N} (\rho'_nA)-z)\geq p\right)\cr
\leq &\sum_{i= 0}^{M-1}\left[\sum_{\zeta\in\mathcal{M}_{B}, |\zeta|\leq n}\P\left(Z_N=\zeta \right) \zeta(J_i)\right]\P\left(\bar{Z}_{n-N}\left(\sqrt{n-N}\rho'_n (A^+_{2\eta}-(1-\ez)u_i)\right)\geq p\right).
\end{align}
The last inequality holds for $\eta>0$ and $n$ large enough, because that for $n$ sufficiently large, for any $z\in J_i$, we have \[\sqrt{n-N} (\rho'_nA)-z\subset\sqrt{n-N}\rho'_n (A_{2\eta}^+-(1-\ez)u_i),\]
where $A^+_{2\eta}=\cup_{x\in A}B(x,2\eta)$ is the $2\eta$-neighbourhood of $A$.
Observe that for any $0\leq i\leq M-1$,
\begin{eqnarray}\label{upbd0}
\sum_{\zeta\in\mathcal{M}_{B}, |\zeta|\leq n}\P\left(Z_N=\zeta\right) \zeta(J_i)
&\leq& n\P\left(|Z_N|\leq n, Z_N(J_i)\geq 1 \right)\nonumber\\
&\leq&  n\E\left[|Z_N|; |Z_N|\leq n\right]\nu_N\left(J_i\right) .
\end{eqnarray}
Again, by \eqref{local},
\begin{equation}\label{upbd1}
\E[|Z_N|; |Z_N|\leq n]\leq  n \P(|Z_N|\leq n) \leq c_3 n^{\chi+1} p_1^N.
\end{equation}
On the other hand, by \eqref{LDPRW} and Chernoff bound,
\begin{equation}\label{upbd2}
\nu_N\left(J_i\right) \leq e^{-N \inf_{u_i\leq u\leq u_{i+1}}\gamma(u/a)}.
\end{equation}
With (\ref{upbdA}) in hand, plugging \eqref{upbd1} and \eqref{upbd2} into \eqref{upbd0} yields that
\begin{multline}\label{upbd0+}
\P(Z_N\in\mathcal{M}_B, |Z_N|<n, \bar{Z}_n(\sqrt{n}\rho_nA)\geq p)\\
\leq c_3 n^{\chi+2}p_1^N\sum_{i=0}^{M-1}e^{-N \inf_{u_i\leq u\leq u_{i+1}}\gamma(u/a)}\P\left(\bar{Z}_{n-N}\left(\sqrt{n-N}\rho'_n (A_{2\eta}^+-(1-\ez)u_i)\right)\geq p\right).
\end{multline}
It remains to bound
\(
\P\left(\bar{Z}_{n-N}(\sqrt{n-N}\rho'_n (A_{2\eta}^+-(1-\ez)u_i))\geq p\right)
\).
Here we are going to use the assumption \eqref{upbdL}. We first consider $A^+_{2\eta}-(1-\ez)u_i$ and the corresponding $I_{A^+_{2\eta}-(1-\ez)u_i}(p)$. According to  (5) of Lemma \ref{LP}, we have for any $x\in\mR$, $$\lambda((\partial A)_{2\eta}^+)+\nu(A-x)\geq \nu(A_{2\eta}^+-x)\geq \nu(A-x).$$ We take $\eta>0$ sufficiently small (i.e., $M$ large enough) so that $\lambda((\partial A)^{2\eta})\leq \delta/2,$ which  ensures that for any $u\in\mR$,
\begin{equation}\label{Aeta}
I_{A-u}(p-\delta/2)\leq I_{A_{2\eta}^+-u}(p)\leq I_{A-u}(p).
\end{equation}
Moreover, by \eqref{nvB}, $\sup_{u_0\leq u\leq u_M}\nu(A^+_{2\eta}-(1-\ez)u)\leq p-\delta/2$. So, we can apply \eqref{upbdL} and obtain that
\begin{align*}
\limsup_{n\rightarrow\infty}\frac{1}{\sqrt{n}}\log\P\left(\bar{Z}_{n-N}(\sqrt{n-N}\rho'_n (A_{2\eta}^+-(1-\ez)u_i))\geq p\right)\leq& -LI_{A^+_{2\eta}-(1-\ez)u_i}(p)\\
\leq & -L I_{A-(1-\ez)u_i}(p-\delta/2).
\end{align*}
Let us now introduce $x_+(p):=\inf\{x\geq0:\nu(A-x)\geq p\}$ and $x_-(p):=\inf\{x\geq0: \nu(A+x)\geq p\}$. Clearly $I_A(p)=x_+(p)\wedge x_-(p)$. For $u\in(u_0,u_M)$,
\[
I_{A-u}(p)=(x_+(p)-u)\wedge(x_-(p)+u).
\]
By (1) of Lemma \ref{LP}, we can choose $\delta>0$ sufficiently small such that $x_\pm(p-\delta/2)\in [x_\pm(p)(1-\ez), x_\pm(p))$. This brings that
\[
I_{A-(1-\ez)u_i}(p-\delta/2)\geq (1-\ez) ((x_+(p)-u_i)\wedge(x_-(p)+u_i)).
\]
It implies that
\[
\limsup_{n\rightarrow\infty}\frac{1}{\sqrt{n}}\log\P\left(\bar{Z}_{n-N}(\sqrt{n-N}\rho'_n (A_{2\eta}^+-u_i))\geq p\right)\leq -L(1-\ez) ((x_+(p)-u_i)\wedge(x_-(p)+u_i)).
\]
As a result, \eqref{upbd0+} entails that
\begin{align}\label{upbdB}
&\limsup_{n\rightarrow\infty}\frac{1}{\sqrt{n}}\log\P(Z_N\in\mathcal{M}_B, |Z_N|<n, \bar{Z}_n(\sqrt{n}\rho_nA)\geq p) \cr
\leq &-(1-\ez)\min_{0\leq i\leq M-1}\left\{a\log \frac{1}{p_1}+a\inf_{u_i\leq u\leq u_{i+1}}\gamma(u/a)+L ((x_+(p)-u_i)\wedge(x_-(p)+u_i))\right\}\cr
\leq & -(1-\ez)\inf_{u_0\leq u\leq u_M}\left\{a\log \frac{1}{p_1}+a\gamma(u/a)+L ((x_+(p)-u)\wedge(x_-(p)+u))\right\}+\eta L,
\end{align}
as $u_{i+1}-u_i=\eta$. Observe that for any $u_0\leq u\leq u_M$,
\[
a\log \frac{1}{p_1}+a\gamma(u/a)+L (x_+(p)-u)\geq \inf_{u\in\mR}\left\{a\log \frac{1}{p_1}+a\gamma(u/a)+L (x_+(p)-u)\right\},
\]
and
\[
a\log \frac{1}{p_1}+a\gamma(u/a)+L (x_-(p)+u)\geq \inf_{u\in\mR}\left\{a\log \frac{1}{p_1}+a\gamma(u/a)+L (x_-(p)+u)\right\}.
\]
Recall that $a=\alpha I_A(p)$ and that the symmetry of the distribution of step size implies that $\gamma(u)=\gamma(-u)$. Therefore,
\begin{eqnarray*}
\ar\ar\inf_{u_0\leq u\leq u_M}\left\{a\log \frac{1}{p_1}+a\gamma(u/a)+L ((x_+(p)-u)\wedge(x_-(p)+u))\right\}\\
\ar\ar\quad\geq  \inf_{u\in\mR}\left\{a\log \frac{1}{p_1}+a\gamma(u/a)+L (x_+(p)-u)\right\} \wedge \inf_{u\in\mR}\left\{a\log \frac{1}{p_1}+a\gamma(u/a)+L (x_-(p)+u)\right\}\\
\ar\ar\quad= \inf_{u\in\mR}\left\{a\log \frac{1}{p_1}+a\gamma(u/a)+Lu\right\}+L(x_+(p)\wedge x_-(p))\\
\ar\ar\quad=F(L) I_A(p),
\end{eqnarray*}
 which, together with (\ref{mainevent})and (\ref{upbdB}), implies (\ref{upbdL0}), by letting $\eta\downarrow0$ and $\ez\downarrow0$ in (\ref{upbdB}).  We complete the proof.
\end{proof}

Now we are prepared to prove \eqref{lsup}.
\begin{proof}[{\bf Proof of \eqref{lsup}}]
We begin with the rough bound in \eqref{rub}. Let $L_0=-\alpha\log p_1$ and  \[
L_k:=F(L_{k-1}),\quad\forall k\geq1.
\]
In view of Lemma \eqref{iteration}, by iteration, we get that for any $k\geq0$,
\[
\limsup_{n\rightarrow\infty}\frac{1}{\sqrt{n}}\log {\mbb P}\left(\bar{Z}_n(\sqrt{n}A)\geq p\right)\leq -I_A(p)L_k.
\]
By Lemma \ref{A1}, $\lim_{k\rightarrow\infty}L_k=\bar{\Lambda}(p_1)$. We thus conclude \eqref{lsup}.

\end{proof}

\subsection{Proof of Theorem \ref{Main02}: $I_A(p)=\infty$}
The idea of the proof is mainly borrowed from Louidor and Perkins \cite[Section 2.3.2]{LP15}. Recall that $J_A(p)=\inf\l\{y: \sup_{x\in {\mbb R}}\nu((A-x)/\sqrt{1-y})\geq p,\, y\in[0,1]\r\}$.

\begin{lemma}\label{lem02low}  Suppose that all assumptions in Theorem \ref{Main02} hold. Then
\begin{eqnarray}\label{lem02low01}
\liminf_{n\rightarrow\infty}\frac{1}{n}\log {\mbb P}\left(\bar{Z}_n(\sqrt{n}A)\geq p\right)\geq J_A(p)\log p_1.
\end{eqnarray}
\end{lemma}
\begin{proof} Since $J_A(p)$ is continuous at $p$, then by (3) in Lemma \ref{LP},
for any $\ez>0$ small enough, we may find $r\in (0, 1)$, $x\in {\mbb R}$ and $\delta>0$, $\eta>0$, such that
$$
J_A(p)-\ez<r< J_A(p)+\ez\quad\text{and}\quad \nu(\cap_{y\in[x-\eta,x+\eta]}(A-y)/\sqrt{1-r})\geq p+\delta.
$$
Set
$$
t_n=rn+x\sqrt{n};\quad m=n-t_n;\quad B_n:=[(x-\eta)\sqrt{n}, (x+\eta)\sqrt{n}].
$$
Observe that
\[
\P( \bar{Z}_n(\sqrt{n}A)\geq p)\geq \P( \bar{Z}_n(\sqrt{n}A)\geq p, Z_{t_n}\in \mathcal{M}_{B_n}, |Z_{t_n}|=1).
\]
Applying Markov property at time $t_n$ implies that
\begin{align*}
\P( \bar{Z}_n(\sqrt{n}A)\geq p)\geq& \P( Z_{t_n}\in \mathcal{M}_{B_n}, |Z_{t_n}|=1)\inf_{y\in B_n}\P(\bar{Z}_{n-t_n}^{\delta_y}(\sqrt{n}A)\geq p)\\
=& p_1^{t_n}\nu_{t_n}(B_n)\inf_{y\in [(x-\eta)\sqrt{n}, (x+\eta)\sqrt{n}]}\P(\bar{Z}_{m}(\sqrt{n} A-y)\geq p)\\
\geq & p_1^{t_n}\nu_{t_n}(B_n)\P\Big(\bar{Z}_m(\sqrt{m}\rho_n\cap_{y\in[x-\eta,x+\eta]}\frac{(A-y)}{\sqrt{1-r}})\geq p\Big)
\end{align*}
where $\rho_n=\frac{\sqrt{n(1-r)}}{\sqrt{m}}\rightarrow 1$. By Lemma \ref{LP},
\[
\left|\nu_{m}\left(\sqrt{m}\rho_n\cap_{y\in[x-\eta,x+\eta]}
\frac{(A-y)}{\sqrt{1-r}}\right)-\nu\left(\cap_{y\in[x-\eta,x+\eta]}\frac{(A-y)}{\sqrt{1-r}}\right)\right|=o_n(1).
\]
So for $n$ large enough,
\[
\nu_{m}\left(\sqrt{m}\rho_n\cap_{y\in[x-\eta,x+\eta]}\frac{(A-y)}{\sqrt{1-r}}\right)\geq \nu\left(\cap_{y\in[x-\eta,x+\eta]}\frac{(A-y)}{\sqrt{1-r}}\right)-\delta/2\geq p+\delta/2.
\]
This, together with Lemma \ref{LP} entails that
\begin{align*}
&\P( \bar{Z}_n(\sqrt{n}A)\geq p)\\
\geq&  p_1^{t_n}\nu_{t_n}(B_n)\P\left(\bar{Z}_m\left(\sqrt{m}\rho_n\cap_{y\in[x-\eta,x+\eta]}\frac{(A-y)}{\sqrt{1-r}}\right)\geq \nu_{m}\left(\sqrt{m}\rho_n\cap_{y\in[x-\eta,x+\eta]}\frac{(A-y)}{\sqrt{1-r}}\right)-\delta/2\right)\\
\geq & p_1^{t_n}\nu_{t_n}(\sqrt{n}[x-\eta,x+\eta])(1+o_n(1)).
\end{align*}
Note that $\nu_{t_n}(\sqrt{n}[x-\eta,x+\eta])=\Theta(1)$ by classical central limit theorem. As a consequence,
\[
\liminf_{n\rightarrow\infty}\frac{1}{n}\log {\mbb P}\left(\bar{Z}_n(\sqrt{n}A)\geq p\right)\geq r\log p_1\geq (J_A(p)+\ez)\log p_1.
\]
We obtain \eqref{lem02low01} by letting $\ez\downarrow 0$.
\end{proof}
\smallskip
\begin{lemma}\label{lem02up}  Suppose that all assumptions in Theorem \ref{Main02} hold. Then for any $\ez>0$ small enough,
\begin{eqnarray}\label{lem02up01}
\limsup_{n\rightarrow\infty}\frac{1}{n}\log {\mbb P}\left(\bar{Z}_n(\sqrt{n}A)\geq p\right)\leq (J_A(p)-\ez)\log p_1.
\end{eqnarray}
\end{lemma}
\begin{proof}
For $\ez\in(0, J_A(p))$ small enough, set
$
t_n=\lfloor(J_A(p)-\ez)n\rfloor.
$
Then
\begin{equation}\label{upbd3}
\P(\bar{Z}_n(\sqrt{n}A)\geq p)=\sum_{\zeta\in \mathcal{M}}\P(\bar{Z}_{n-t_n}^{\zeta}(\sqrt{n}A)\geq p)\P(Z_{t_n}=\zeta).
\end{equation}
By the definition of $J_A(p)$, there exists $\dz>0$ such that for $\ez'\in [\ez, 2\ez]$,
\begin{eqnarray}
\sup_{y\in\mR}\nu\left(\frac{A-y}{\sqrt{1-J_A(p)+\ez'}}\right)\leq p-\dz.
\end{eqnarray}
(1) and (5) of Lemma \ref{LP} show that for $n$ large enough,
$$
\frac{1}{|\zeta|}\sum_{y\in \zeta}\nu_{n-t_n}(\sqrt{n}A-y)\leq \frac{1}{|\zeta|}\sum_{y\in \zeta}\nu(\frac{\sqrt{n}}{\sqrt{n-t_n}}A-\frac{y}{\sqrt{n-t_n}})+\delta/2\leq p-\delta/2.
$$
This implies that
\begin{align*}
\P\left(\bar{Z}_{n-t_n}^{\zeta}(\sqrt{n}A)\geq p\right)\leq &\P\left(\bar{Z}_{n-t_n}^{\zeta}(\sqrt{n}A)\geq \frac{1}{|\zeta|}\sum_{y\in \zeta}\nu_{n-t_n}(\sqrt{n}A-y)+\dz/2\right),
\end{align*}
which by Lemma \ref{LP+} is less than $C_1e^{-C_2\delta^2|\zeta|/4}$. Going back to \eqref{upbd3} and using \eqref{local}, we have
\begin{align}\label{lem02up02}
\P(\bar{Z}_n(\sqrt{n}A)\geq p)
\leq &\sum_{\zeta\in \mathcal{M}}C_1e^{-C_2\delta^2|\zeta|/4}\P(Z_{t_n}=\zeta)\nonumber\\
= &\sum_{k=1}^\infty C_1\P(|Z_{t_n}|=k)e^{-c_4 k}\nonumber\\
\leq & c_5 \sum_{k=1}^\infty p_1^{t_n} k^{\chi-1}e^{-c_4 k}\leq c_6 p_1^{t_n}.
\end{align}
This yields (\ref{lem02up01}) immediately.
\end{proof}

Theorem \ref{Main02} follows directly from the lemmas \ref{lem02low} and \ref{lem02up}.

\section{B\"{o}ttcher case}

In this section, we suppose that $p_0=p_1=0$. As we claimed in the introduction, different tail distributions of step size bring out different regimes. To obtain Theorems \ref{Main03} and \ref{Main04}, we need to treat the sub-exponential and the super-exponential decaying tails differently.

\subsection{Proof of Theorem \ref{Main03}: step size has Weibull tail distribution}

\subsubsection{Proof when step size has (sub)-exponential decay}

In this section, we assume that $2\leq b<B\leq \infty$, $\P(X>z)=\Theta(1) e^{-\lambda z^\alpha}$ as $z\rightarrow\infty$ with $\alpha\in(0,1]$, $\lambda>0$ and also that $I_A(p)<\infty$ and $I_A(\cdot)$ is continuous at $p=\nu(A)+\Delta$. We are devoted to proving
\begin{equation}\label{limsubex}
\lim_{n\rightarrow\infty}\frac{1}{n^{\alpha/2}}\log \P(\bar{Z}_n(\sqrt{n} A)-\nu(A)\geq \Delta)=-\lambda I_A(p)^\alpha.
\end{equation}

\paragraph{Lower bound of \eqref{limsubex}:}
Suppose that at the first generation, the root gives birth to exactly $b$ children, denoted by $\rho_1,\cdots,\rho_b$. Moreover, suppose that their positions are $S_{\rho_1}\in [(x-\eta)\sqrt{n}, (x+\eta)\sqrt{n}]$, $S_{\rho_i}\in [-M, M]$ for $2\leq i\leq b$. Here we take $M>0$ such that
\[
\P(|X|\leq M)\geq 1/2.
\]
As $I_A(\cdot)$ is continuous and finite at $p$, for any $\ez>0$ small enough, there exist $x\in\mR$ and $\eta>0,\delta>0$ such that
\[
\inf_{y\in[x-\eta,x+\eta]}\nu(A-y)\geq p+\delta,\quad |x|=I_A(p)+\ez.
\]
By Lemma \eqref{lemloc}, given $\mathbf{E}:=\left\{|Z_1|=b, S_{\rho_1}\in [(x-\eta)\sqrt{n}, (x+\eta)\sqrt{n}], S_{\rho_i}\in [-M, M], \forall i=2,\cdots,b\right\}$,
\[
\liminf_{n\rightarrow\infty}\bar{Z}_n(\sqrt{n}A)\geq \frac{(p+\delta)W_1+\sum_{i=2}^b \nu(A)W_i}{\sum_{i=1}^b W_i},
\]
where $W_i,\ i=1, 2,\cdots, b$ are i.i.d. copies of $W$. This shows that
\begin{align*}
\liminf_{n\rightarrow\infty}\P(\bar{Z}_n(\sqrt{n} A)\geq p\vert \mathbf{E})\geq &\P\left(\liminf_{n\rightarrow\infty} \bar{Z}_n(\sqrt{n}A)\geq p+\delta/2\big{\vert}\mathbf{E}\right)\\
\geq &\P\left(\frac{(p+\delta)W_1+\sum_{i=2}^b \nu(A)W_i}{\sum_{i=1}^b W_i}\geq p+\delta/2\right)\\=:& C_{A,p,\delta, b}.
\end{align*}
Since $B>b$,  $W$ has a continuous positive density on $(0, \infty)$; see Athreya and Ney \cite[Chapter II, Lemma 2]{AN72}. So $C_{A,p,\delta, b}$ is a positive real number. Consequently,
\begin{align*}
 \P(\bar{Z}_n(\sqrt{n} A)-\nu(A)\geq \Delta)\geq & \P\left( \bar{Z}_n(\sqrt{n} A)\geq p\vert \mathbf{E}\right)\P(\mathbf{E})\\
 \geq & \Theta(1)p_b e^{-\lambda (|x|-\eta)^{\alpha}n^{\alpha/2}}(1+o_n(1))\left(\frac{1}{2}\right)^{b-1}.
\end{align*}
Taking limits yields that
\begin{align*}
\liminf_{n\rightarrow\infty}\frac{1}{n^{\alpha/2}}\log \P(\bar{Z}_n(\sqrt{n} A)-\nu(A)\geq \Delta)\geq -\lambda (I_A(p)+\ez-\eta)^{\alpha}.
\end{align*}
Letting $\ez\downarrow0$ and $\eta\downarrow 0$ gives that
\[
\liminf_{n\rightarrow\infty}\frac{1}{n^{\alpha/2}}\log \P(\bar{Z}_n(\sqrt{n} A)-\nu(A)\geq \Delta)\geq -\lambda I_A(p)^{\alpha}.
\]

\paragraph{Upper bound of \eqref{limsubex}:}
 Take an intermediate generation $t_n=\lfloor t\log n\rfloor$ with some $t>0$. Let $B_n=[(-I_A(p)+\ez)\sqrt{n}, (I_A(p)-\ez)\sqrt{n}]$ with $\ez>0$ small enough. Note that there exists $\delta>0$ such that
\begin{equation}\label{nuB}
\sup_{x\in [-I_A(p)+\ez, I_A(p)-\ez]}\nu(A-x)\leq p-\delta.
\end{equation}
Observe that
\begin{align}\label{Main03up}
&\P\left(\bar{Z}_n(\sqrt{n} A)\geq p\right)\cr
&\quad\leq  \P(Z_{t_n}(B_n^c)\geq1)+\P(Z_{t_n}\in \mathcal{M}_{B_n}, \bar{Z}_n(\sqrt{n} A)\geq p)\cr
&\quad\leq  2\P(Z_{t_n}((I_A(p)-\ez)\sqrt{n},\infty)\geq 1)+ \sum_{\zeta\in \mathcal{M}_{B_n}}\P(Z_{t_n}=\zeta)\P( \bar{Z}^\zeta_{n-t_n}(\sqrt{n} A)\geq p).
\end{align}
On the one hand, by Markov inequality,
\begin{equation}\label{nuBup}
\P(Z_{t_n}((I_A(p)-\ez)\sqrt{n},\infty)\geq 1)\leq \E[Z_{t_n}((I_A(p)-\ez)\sqrt{n},\infty)]= m^{t_n}\nu_{t_n}((I_A(p)-\ez)\sqrt{n},\infty).
\end{equation}
It is known (see \cite{Na79}) that
\beqlb\label{nuBup++}
\nu_{t_n}((I_A(p)-\ez)\sqrt{n},\infty)\begin{cases}= e^{-\lambda (I_A(p)-\ez)^{\alpha}n^{\alpha/2}+o(n^{\alpha/2})}, & \alpha<1;\cr
\leq e^{-(\lambda-\ez)(I_A(p)-\ez)\sqrt{n}}\E[e^{(\lambda-\ez)X}]^{t_n}, &\alpha=1,
\end{cases}
\eeqlb
where for $\alpha=1$, we use Markov inequality again by noting that $\E[e^{-(\lambda-\ez)X}]<\infty$ for $\ez\in(0,\lambda)$.
On the other hand, for any $\zeta\in \mathcal{M}_{B_n}$, because of \eqref{nuB}, one has
\[
\frac{1}{|\zeta|}\sum_{x\in\zeta}\nu(A-\frac{x}{\sqrt{n}})\leq p-\delta.
\]
Again by (6) of Lemma \ref{LP}, for $n$ large enough,
\[
\frac{1}{|\zeta|}\sum_{x\in \zeta}\nu_{n-t_n}(\sqrt{n}A-x)\leq p-\delta/2.
\]
This, combined with Lemma \ref{LP+}, implies that
\begin{align}\label{nuBup+}
\P( \bar{Z}^\zeta_{n-t_n}(\sqrt{n} A)\geq p)\leq &\P\left(\bar{Z}^\zeta_{n-t_n}(\sqrt{n} A)\geq \frac{1}{|\zeta|}\sum_{x\in \zeta}\nu_{n-t_n}(\sqrt{n}A-x)+\delta/2\right)\nonumber\\
\leq & C_1 e^{-C_2|\zeta|\delta^2/4}.
\end{align}
In view of \eqref{nuBup}, \eqref{nuBup++} and \eqref{nuBup+}, \eqref{Main03up} becomes
\begin{eqnarray*}
\ar\ar\P\left(\bar{Z}_n(\sqrt{n} A)\geq p\right)\cr\ar\ar\quad\leq  2m^{t_n}e^{-(\lambda-\ez)(I_A(p)-\ez)^{\alpha}n^{\alpha/2}+o(n^{\alpha/2})}+C_1\sum_{\zeta\in \mathcal{M}_{B_n}}\P(Z_{t_n}=\zeta)e^{-C_2|\zeta|\delta^2/4}\cr\ar\ar\quad
\leq 2m^{t_n}e^{-(\lambda-\ez) (I_A(p)-\ez)^{\alpha}n^{\alpha/2}+o(n^{\alpha/2})}+C_1e^{-c_7b^{t_n}},
\end{eqnarray*}
as $|Z_{t_n}|\geq b^{t_n}$. We take $t>\frac{\alpha}{2\log b}$ so that $b^{t_n}\gg n^{\alpha/2}$ and hence
\[
\P( \bar{Z}^\zeta_{n-t_n}(\sqrt{n} A)\geq p)\leq e^{-(\lambda-\ez) (I_A(p)-\ez)^{\alpha}n^{\alpha/2}+o(n^{\alpha/2})}.
\]
We thus obtain that
\[
\limsup_{n\rightarrow\infty}\frac{1}{n^{\alpha/2}}\log \P( \bar{Z}^\zeta_{n-t_n}(\sqrt{n} A)\geq p)\leq -(\lambda-\ez) (I_A(p)-\ez)^{\alpha},
\]
which, together with the lower bound above, concludes \eqref{limsubex}.
\subsubsection{Proof when step size has super-exponential decay}\label{proofsupex}

In this section, we assume that the tail distribution of step size is $\P(X>x)=\Theta(1) e^{-\lambda x^{\alpha}}$ with $\alpha>1$. The embedding tree is assumed to be random with $2\leq b<B\leq +\infty$. We are going to prove the following convergence: for $p\in (\nu(A), 1)$ such that $I_A(p)$ is continuous and finite, we have
\begin{equation}\label{limsuperex}
\lim_{n\rightarrow\infty}\frac{(\log n)^{\alpha-1}}{n^{\alpha/2}}\log\P\left(\bar{Z}_n(\sqrt{n}A)\geq p\right) = -\left(\frac{2\log b\log B}{\alpha(\log B-\log b)}\right)^{\alpha-1} \lambda I_A(p)^{\alpha},
\end{equation}
with the convention that  $\frac{2\log b\log B}{\alpha(\log B-\log b)}=\frac{2\log b}{\alpha}$ if $B=\infty$.

\paragraph{Lower bound of \eqref{limsuperex}:}
According to the definition of $I_A(p)$, for any $\delta>0$ small enough, there exist $x_0\in\mR$, $\ez>0$ and $\eta>0$ such that
\[
I_A(p)< |x_0|\leq I_A(p)+\ez, \ \inf_{y\in[x_0-\eta,x_0+\eta]}\nu(A-y)\geq p+\delta.
\]
Take an integer $d>b$ such that $p_{d}>0$.  (If $B<\infty$,  then we choose $d=B$. If $B=\infty$, we will let $d\rightarrow\infty$ later.)  Let $t_n=\lfloor t_1\log n-t_2\log\log n\rfloor$ with some $t_1, t_2>0$. Then for sufficiently large $n$, let
\[
s_n=\left\lfloor\frac{\log\log n+t_n\log b}{\log d}\right\rfloor,\quad \Upsilon=\frac{2\log b\log d}{\alpha(\log d-\log b)}.
\]
Let us construct a tree ${\bf t}$ of height $t_n$ in the following way.  First, ${\bf t}_{t_n-s_n}:=\{v\in {\bf t}: |v|\leq t_s-s_n\}$ is a $b$-regular tree. Using Neveu's notation \cite{Ne86}, let $\mathbb{U}:=\cup_{n\geq 1}\mathbb{N}_+^n\cup\{\rho\}$ be the infinite Ulam-Harris tree, to code the vertices.  Here denote $u^*=(1, \cdots, 1)$ to be the first individual of the $(t_n-s_n)$-th generation in the lexicographic order. Next,  ${\bf t}(u^*)$ is a $d$-regular tree and $\{{\bf t}(u): u\neq u^*, |u|=t_n-s_n\}$ are all $b$-regular trees, where for any $u\in {\bf t}$, ${\bf t}(u):=\{v\in{\bf t}: u\preceq v\}$ is the subtree of ${\bf t}$ rooted at $u$.
Recall from the very beginning of this paper that $\cal T$ is the embedded Galton-Watson tree. Let ${\cal T}_{t_n}=\{u\in {\cal T}: |u|\leq t_n\}$. Define the following event
\beqlb\label{EVENTE}
\mathcal{E}_{t_n,b,d}=\{{\cal T}_{t_n}={\bf t}, S_u\in [(x_0-\eta)\sqrt{n}, (x_0+\eta)\sqrt{n}], \text{ for all }|u|=t_n\textrm{ s.t. }u\in {\bf t}(u^*)\},
\eeqlb
which means that all the descendants of $u^*$ at the $t_n$-th generation are positioned in the interval $[(x_0-\eta)\sqrt{n}, (x_0+\eta)\sqrt{n}]$. It follows immediately that
\begin{equation}
\P\left(\bar{Z}_n(\sqrt{n}A)\geq p\right)\geq \P\left(\bar{Z}_n(\sqrt{n}A)\geq p\vert \mathcal{E}_{t_n,b,d}\right)\P(\mathcal{E}_{t_n,b,d})
\end{equation}
and that
\begin{equation}\label{keyevent}
\P(\mathcal{E}_{t_n,b,d})\geq p_b^{b^{t_n}}p_d^{d^{s_n}}\P^{\bf t}(S_u\in [(x_0-\eta)\sqrt{n}, (x_0+\eta)\sqrt{n}], \forall u\in {\bf t}(u^*), |u|= t_n),
\end{equation}
where $\P^{\bf t}=\P(\cdot| {\cal T}_{t_n}={\bf t})$. To bound the probability on the R. H. S. of \eqref{keyevent}, let us take the following labels (step sizes):
\[
X_v\in \left[(x_0-\frac{\eta}{2})\frac{\sqrt{n}}{t_n-s_n},(x_0+\frac{\eta}{2})\frac{\sqrt{n}}{t_n-s_n}\right],\quad \forall \, \rho\prec v\preceq  u^*;
\]
and
\[
X_v\in [-M, M],\quad  \forall\,  v\in {\bf t}(u^*) \textrm{ and } |v|\leq t_n,
\]
where $M$ is a fixed real number such that $\P(X\in[-M,M])\geq 1/2$. Observe that for $n$ large enough, for any $u\in{\bf t}(u^*) $ s.t. $|u|=t_n$,
\[
S_u\in \left[(x_0-\frac{\eta}{2})\sqrt{n}-Ms_n, (x_0+\frac{\eta}{2})\sqrt{n}+Ms_n\right]\subset [(x_0-\eta)\sqrt{n}, (x_0+\eta)\sqrt{n}].
\]
As a consequence,
\begin{align*}
&\P^{\bf t}(S_u\in (x_0-\eta)\sqrt{n}, (x_0+\eta)\sqrt{n}], \forall  u\in {\bf t}(u^*), |u|= t_n)\\
\geq &\P\left(X\in \left[(x_0-\frac{\eta}{2})\frac{\sqrt{n}}{t_n-s_n},(x_0+\frac{\eta}{2})\frac{\sqrt{n}}{t_n-s_n}\right]\right)^{t_n-s_n}\times \prod_{1<k\leq s_n}\P(X\in[-M,M])^{d^k}\\
\geq & \left(\frac{1}{2}\right)^{2d^{s_n}}c_8^{t_n-s_n}\exp\left\{-\lambda(t_n-s_n)^{1-\alpha}\left(|x_0|-\frac{\eta}{2}\r)^{\alpha}n^{\alpha/2}\right\}.
\end{align*}
Here we take $t_1=\frac{\alpha}{2\log b}$ and $t_2=\frac{2\alpha}{\log b}$ so that
\[
\frac{n^{\alpha/2}}{(t_n-s_n)^{\alpha-1}}=(1+o_n(1))\Upsilon^{\alpha-1} \frac{n^{\alpha/2}}{(\log n )^{\alpha-1}}\gg d^{s_n}\gg b^{t_n}.
\]
Therefore, by \eqref{keyevent},
\begin{equation}\label{Etbd}
\P(\mathcal{E}_{t_n,b,d})\geq \exp\l\{-\lambda \l(|x_0|-\frac{\eta}{2}\r)^{\alpha}\Upsilon^{\alpha-1}\frac{n^{\alpha/2}}{\l(\log n \r)^{\alpha-1}}+o\l(\frac{n^{\alpha/2}}{(\log n )^{\alpha-1}}\r)\r\}.
\end{equation}
It remains to bound $ \P\left(\bar{Z}_n(\sqrt{n}A)\geq p\vert \mathcal{E}_{t_n,b,d}\right)$. Define
\[
\mathcal{M}_0=\l\{\zeta\in {\cal M}: |\zeta|=d^{s_n}+b^{t_n}-b^{s_n}, \zeta( [(x_0-\eta)\sqrt{n}, (x_0+\eta)\sqrt{n}])\geq d^{s_n}\r\}.
\]
For any $\zeta\in\mathcal{M}_0 $, by considering only the particles in $[(x_0-\eta)\sqrt{n}, (x_0+\eta)\sqrt{n}]$,
\begin{align*}
\frac{1}{|\zeta|}\sum_{x\in \zeta}\nu(\sqrt{n}A-x)\geq &\frac{d^{s_n}}{d^{s_n}+b^{t_n}-b^{s_n}}\inf_{y\in[x_0-\eta,x_0+\eta]}\nu(A-y)\\
\geq &\frac{\log n}{\log n+1}(p+\delta),
\end{align*}
which, combining with (5) of Lemma \ref{LP}, implies that for $n$ large enough,
\[
\frac{1}{|\zeta|}\sum_{x\in \zeta}\nu_{n-t_n}(\sqrt{n}A-x)\geq p+\delta/2.
\]
This means that for any $\zeta\in\mathcal{M}_0 $, \[\l\{\bar{Z}_{n-t_n}^\zeta(\sqrt{n}A)\geq \frac{1}{|\zeta|}\sum_{x\in \zeta}\nu_{n-t_n}(\sqrt{n}A-x)-\delta/2\r\}\subset\{\bar{Z}_{n-t_n}^\zeta(\sqrt{n}A)\geq p\},\] as $p=\nu(A)+\Delta$.  Note that $Z_{t_n}\in\mathcal{M}_0$ given  $\mathcal{E}_{t_n,b,d} $.
Then by Lemma \ref{LP+}, we have
\begin{align*}
&\P(\bar{Z}_n(\sqrt{n}A)\geq p\vert \mathcal{E}_{t_n,b,d})\\
&\geq  \P\left(\bar{Z}_{n-t_n}^\zeta(\sqrt{n}A)\geq \frac{1}{|\zeta|}\sum_{x\in \zeta}\nu_{n-t_n}(\sqrt{n}A-x)-\delta/2\right)\Bigg\vert_{\zeta=Z_{t_n}\in\mathcal{M}_0}\\
&\geq  1-C_1 e^{-C_2d^{s_n}\delta^2/4}\geq \frac{1}{2},
\end{align*}
for all $n$ large enough. This implies that
\[
\P(\bar{Z}_n(\sqrt{n}A)\geq p)\geq \frac{1}{2}\P(\mathcal{E}_{t_n,b,d}),
\]
which,  together with \eqref{Etbd}, gives
\beqlb\label{limsuperexlow}
\liminf_{n\rightarrow\infty}\frac{(\log n)^{\alpha-1}}{n^{\alpha/2}}\log \P(\bar{Z}_n(\sqrt{n}A)\geq p)\geq -\lambda\l(|x_0|-\frac{\eta}{2}\r)^{\alpha}\Upsilon^{\alpha-1}.
\eeqlb
Then we get the lower bound by letting $\delta\downarrow 0$ and $d\uparrow B$.

\paragraph{Upper bound of \eqref{limsuperex}:}  Again, by the definition of $I_A(p)$, for any $\delta>0$ small enough, there exist $\eta>0$ such that
\[
\sup_{|y|\leq I_A(p)-\eta}\nu(A-y)\leq p-\delta.
\]
Let $B_n:=[(-I_A(p)+\eta)\sqrt{n}, (I_A(p)-\eta)\sqrt{n}]$, $t_n=\lfloor\frac{\alpha\log b}{2}\log n\rfloor$. Observe that for any $\zeta\in\mathcal{M}$,
\begin{align*}
\frac{1}{|\zeta|}\sum_{x\in\zeta}\nu(A-\frac{x}{\sqrt{n}})=&\frac{1}{|\zeta|}\sum_{x\in\zeta\cap B_n} \nu(A-\frac{x}{\sqrt{n}})+\frac{1}{|\zeta|}\sum_{x\in\zeta\cap B_n^c} \nu(A-\frac{x}{\sqrt{n}})\\
\leq & p-\delta+ \frac{\zeta(B_n^c)}{|\zeta|},
\end{align*}
which is less than $p-\delta/2$ as soon as $\frac{\zeta(B_n^c)}{|\zeta|}\leq \delta/2$. Further, by (5) of Lemma \ref{LP}, for all $n$ large enough,
\begin{equation}\label{nutp}
{\cal M}_1:=\l\{\zeta\in{\cal M}: \frac{\zeta(B_n^c)}{|\zeta|}\leq \delta/2\r\}\subset\l\{\zeta\in{\cal M}: \frac{1}{|\zeta|}\sum_{x\in\zeta}\nu_{n-t_n}(\sqrt{n}A-x)\leq p-\delta/4\r\}.
\end{equation}
By conditioning on $\{Z_{t_n}=\zeta\}$ for any $\zeta\in{\cal M}_1$, we observe that
\begin{align*}
\P(\bar{Z}_n(\sqrt{n}A)\geq p)\leq & \P(\bar{Z}_{t_n}(B_n^c)>\delta/2)+\P(\bar{Z}_n(\sqrt{n}A)\geq p,\bar{Z}_{t_n}(B_n^c)\leq \delta/2)\\
=&\P(\bar{Z}_{t_n}(B_n^c)>\delta/2)+\sum_{\zeta\in \mathcal{M}_1}\P(Z_{t_n}=\zeta)\P\left(\bar{Z}^\zeta_{n-t_n}(\sqrt{n}A)\geq p\right),
\end{align*}
which, by \eqref{nutp}, is bounded by
\[
\P(\bar{Z}_{t_n}(B_n^c)>\delta/2)+\sum_{\zeta\in \mathcal{M}}\P(Z_{t_n}=\zeta)\P\left(\bar{Z}^\zeta_{n-t_n}(\sqrt{n}A)\geq \frac{1}{|\zeta|}\sum_{x\in\zeta}\nu_{n-t_n}(\sqrt{n}A-x)+\delta/4\right).
\]
Note that  $|Z_{t_n}|\geq b^{t_n}$. In view of Lemma \ref{LP+},
\[
\sum_{\zeta\in \mathcal{M}}\P(Z_{t_n}=\zeta)\P\left(\bar{Z}^\zeta_{n-t_n}(\sqrt{n}A)\geq \frac{1}{|\zeta|}\sum_{x\in\zeta}\nu_{n-t_n}(\sqrt{n}A-x)+\delta/4\right)\leq C_1 e^{-c_9 b^{t_n}}.
\]
Since $\P(\bar{Z}_{t_n}(B_n^c)>\delta/2)\leq \P(Z_{t_n}(B_n^c)\geq \delta b^{t_n}/2)$, then
\begin{equation}\label{Main03up+}
\P(\bar{Z}_n(\sqrt{n}A)\geq p)\leq \P(Z_{t_n}(B_n^c)\geq \delta b^{t_n}/2)+C_1 e^{-c_9 b^{t_n}}.
\end{equation}
It remains to bound $\P(Z_{t_n}(B_n^c)\geq \delta b^{t_n}/2)$, which will be investigated separately in two cases: $B=\infty$ and $B<\infty$.

\subparagraph{First case: $B=\infty$.}

Note that by Markov inequality and symmetry of $X$,
\begin{align*}
\P(Z_{t_n}(B_n^c)\geq \delta b^{t_n}/2) \leq & \P(Z_{t_n}(B_n^c)\geq 1)\\
\leq & m^{t_n}\nu_{t_n}(B_n^c)\leq 2 m^{t_n}\nu_{t_n}\left(\Big[(I_A(p)-\eta)\sqrt{n},\infty\Big)\right).
\end{align*}
Then (1) of Lemma \ref{A2} implies that
\[
\limsup_{n\rightarrow\infty}\frac{(\log n)^{\alpha-1}}{n^{\alpha/2}}\log \P(Z_{t_n}(B_n^c)\geq \delta b^{t_n}/2) \leq -\lambda \left(\frac{2\log b}{\alpha}\right)^{\alpha-1}(I_A(p)-\eta)^\alpha.
\]
As $b^{t_n}\geq \frac{n^{\alpha/2}}{b}\gg \frac{n^{\alpha/2}}{t_n^{\alpha-1}}$, in view of \eqref{Main03up+}, we get
\[
\limsup_{n\rightarrow\infty}\frac{(\log n)^{\alpha-1}}{n^{\alpha/2}}\log\P(\bar{Z}_n(\sqrt{n}A)\geq p)\leq -\lambda \left(\frac{2\log b}{\alpha}\right)^{\alpha-1}(I_A(p)-\eta)^\alpha,
\]
which, with the help of \eqref{limsuperexlow},
 proves \eqref{limsuperex} in the case of $B=\infty$ by letting $\eta\downarrow 0$.
\subparagraph{Second case: $b<B<\infty$.}
The proof will be divided into three subparts.\\
{\it Subpart 1:} Recall that $t_n=\left\lfloor\frac{\alpha}{2\log b}\log n\right\rfloor$. Let $s_n=\l\lfloor\frac{t_n\log b -2\alpha\log\log n}{\log B} \r\rfloor$. For $n$ large enough, we have $\delta b^{t_n}/4\geq B^{s_n}$.
Observe that
\[
\P(Z_{t_n}(B_n^c)\geq \delta b^{t_n}/2) \leq 2\P\left(Z_{t_n}((I_A(p)-\eta)\sqrt{n},\infty)\geq B^{s_n}\right).
\]
Recall that up to the $t_n$-th generation, the genealogical tree ${\cal T}_{t_n}$ is Galton-Watson. Set $I(n)=(I_A(p)-\eta)\sqrt{n}$. Then
\begin{equation*}
\P(Z_{t_n}((I_A(p)-\eta)\sqrt{n},\infty)\geq B^{s_n})=\sum_{{\bf t}}\P({\cal T}_{t_n}={\bf t})\P^{\bf t}\l( \sum_{|u|=t_n,\, u\in{\bf t}}1_{\{S_u>I(n)\}}\geq B^{s_n}\r).
\end{equation*}
Observe that
\beqlb\label{comboundset}
\l\{\sum_{|u|=t_n,\, u\in{\bf t}}1_{\{S_u>I(n)\}}\geq B^{s_n}\r\}\subset \l\{\bigcup_{{\cal J}\subset {\bf t}_{t_n},\, |{\cal J}|=B^{s_n} } \bigcap_{u\in \J} \{S_u>I(n)\}\r\},
\eeqlb
where ${\bf t}_{t_n}=\{u\in {\bf t}: |u|=t_n\}$. This yields that
\begin{equation}\label{Main03upkey}
\P(Z_{t_n}((I_A(p)-\eta)\sqrt{n},\infty)\geq B^{s_n})\leq \sum_{{\bf t}}\P({\cal T}={\bf t})\sum_{{\cal J}\subset {\bf t}_{t_n},\, |{\cal J}|=B^{s_n}}\P^{\bf t}\l(\bigcap_{u\in \J} \{S_u>I(n)\} \r)
\end{equation}
We claim that for any ${\bf t}$ and  ${\cal J}\subset {\bf t}_{t_n}$ with $|{\cal J}|=B^{s_n}$,
\begin{equation}\label{combound}
\P^{\bf t}\l( \bigcap_{u\in \J} \{S_u>I(n)\}\r)\leq   c_{10} t_n B^{s_n}e^{-\lambda n^\alpha}+ (c_{10}n)^{t_nB^{s_n}} \exp\l\{- \frac{\lambda\left((I_A(p)-\eta)\sqrt{n}-t_n\right)^{\alpha}}
{\l(t_n-s_n+\frac{1}{B^{1/(\alpha-1)}-1}\r)^{\alpha-1}}\r\}.
\end{equation}
\eqref{combound} will be proved in {\it Subpart 2}. Notice that
\[
\#\l\{{\cal J}\subset {\bf t}_{t_n},\, |{\cal J}|=B^{s_n} \r\}={|{\bf t}_{t_n}|\choose  B^{s_n}}\leq {B^{t_n}\choose  B^{s_n}},
\]
which, together with \eqref{combound} and \eqref{Main03upkey}, gives
\begin{align*}
\P(Z_{t_n}(B_n^c)\geq \delta b^{t_n}/2)\leq {B^{t_n}\choose  B^{s_n}}\times\left(
c_{10}t_n B^{s_n}e^{-\lambda n^\alpha}+
 (c_{10}n)^{t_nB^{s_n}} \exp\l\{- \frac{\lambda\left((I_A(p)-\eta)\sqrt{n}-t_n\right)^{\alpha}}
 {\l(t_n-s_n+\frac{1}{B^{1/(\alpha-1)}-1}\r)^{\alpha-1}}\r\}\right).
\end{align*}
Note that $B^{s_n}\leq \frac{n^{\alpha/2}}{(\log n)^{2\alpha}}$ and
\[
 {B^{t_n}\choose  B^{s_n}}=\frac{(B^{t_n})!}{(B^{s_n})!(B^{t_n}-B^{s_n})!}\leq \frac{B^{t_n B^{s_n}}}{(B^{s_n})!}\leq B^{t_n B^{s_n}}.
\]
 Then
\begin{align*}
t_n B^{s_n}\log n=&O(1)\frac{n^{\alpha/2}}{(\log n)^{2\alpha-2}}\ll \frac{n^{\alpha/2}}{(\log n)^{\alpha-1}},\\
 \frac{\left((I_A(p)-\eta)\sqrt{n}-t_n\right)^{\alpha}}{\l(t_n-s_n+\frac{1}{B^{1/(\alpha-1)}-1}\r)^{\alpha-1}}=& \frac{n^{\alpha/2}}{(\log n)^{\alpha-1}}\left(\Upsilon^{\alpha-1}(I_A(p)-\eta)^\alpha+o_n(1)\right),
\end{align*}
where $\Upsilon=\frac{2\log b\log B}{\alpha(\log B-\log b)}$. So,
\begin{equation*}
\P(Z_{t_n}(B_n^c)\geq \delta b^{t_n}/2) \leq\exp\left\{-\frac{\lambda n^{\alpha/2}}{(\log n)^{\alpha-1}}\left(\Upsilon^{\alpha-1}(I_A(p)-\eta)^\alpha+o_n(1)\right)\right\}.
\end{equation*}
In view of \eqref{Main03up+}, we conclude that
\[
\limsup_{n\rightarrow\infty}\frac{(\log n)^{\alpha-1}}{n^{\alpha/2}}\log\P(\bar{Z}_n(\sqrt{n}A)\geq p)\leq -\lambda \Upsilon^{\alpha-1}(I_A(p)-\eta)^a,
\]
which, by letting $\eta\rightarrow0$, gives what we need.
.

{\it Subpart 2:} This subpart is devoted to demonstrating \eqref{combound}. For any $\J$, define ${\bf t}_{\J}=\{v\in{\bf t}: \rho\prec v\preceq u, u\in \J\}$. One sees that
\begin{align}\label{maxPt}
 &\text{L.H.S. of }\eqref{combound}\leq\P^{\mathbf{t}}\l(\sum_{\rho\prec v\preceq u}|X_v|\geq I(n), \forall u\in \J\r)\nonumber\\
\leq & \P^{\mathbf{t}}\l(\sup_{v\in {\bf t}_\J}|X_v|\geq n\r)+\P^{\mathbf{t}}\l(\sup_{v\in {\bf t}_\J}|X_v|\leq n; \sum_{\rho\prec v\preceq u}|X_v|\geq I(n), \forall u\in \J\r).
\end{align}
 It follows from the tail distribution of $X$ that there exists $c_{11}\geq 1$ such that
\begin{equation*}
\P(|X|\geq x)\leq c_{11} e^{-\lambda x^\alpha}, \forall x\geq0.
\end{equation*}
As a consequence, we have
\begin{equation}\label{combound02}
 \P^{\mathbf{t}}\l(\sup_{v\in {\bf t}_\J}|X_v|\geq n\r)\leq |{\bf t}_\J|\P(|X|\geq n)\leq c_{11}t_n B^{s_n}e^{-\lambda n^\alpha}.
\end{equation}
Meanwhile,
\begin{align}\label{combound01}
&\P^{\mathbf{t}}\l(\sup_{v\in {\bf t}_\J}|X_v|\leq n; \sum_{\rho\prec v\preceq u}|X_v|\geq I(n), \forall u\in \J\r)\nonumber\\
\leq & \sum_{x_v\in {\mbb N}\cap[0,n), v\in {\bf t}_\J} \P^{\mathbf{t}} \l( \bigcap_{v\in {\bf t}_\J}\l\{|X_v|\in [x_v, x_v+1]\r\}\r)1_{\l\{\min\limits_{u\in\J}\sum\limits_{\rho\prec v\preceq u}(x_v+1)\geq I(n)\r\}} \nonumber\\
\leq & \sum_{x_v\in {\mbb N}\cap[0,n), v\in {\bf t}_\J} {c_{11}}^{t_n B^{s_n}}e^{-\lambda\sum_{v\in\mathbf{t}_\J} x_v^\alpha}1_{\l\{\min\limits_{u\in\J}\sum\limits_{\rho\prec v\preceq u}x_v\geq I_1(n)\r\}},
\end{align}
where $I_1(n)=I(n)-t_n$.
We need to bound R.H.S. of \eqref{combound01}. To end this, we {\bf CLAIM} that from $\{x_v, v\in\mathbf{t}_\J\}$, one can construct a rooted deterministic tree $\mathbf{t}_\J^*$ with labels $x_\rho=0$ and $\{x_v^*, v\in\mathbf{t}_\J^*\setminus\{\rho\}\}\subset\{x_v, v\in\mathbf{t}_\J\}$ such that
\begin{equation}\label{tstar}
\sum_{v\in\mathbf{t}_\J} x_v^\alpha \geq \sum_{v\in\mathbf{t}_\J^*} (x_v^*)^\alpha,\quad \min_{|u|=t_n, u\in \mathbf{t}_\J^*}\sum_{\rho\prec v\preceq u}x_v^*\geq \min_{ u\in \J}\sum_{\rho\prec v\preceq u}x_v\geq I_1(n),
\end{equation}
and that $\mathbf{t}_\J^*$ contains a single branch up to the generation $t_n-s_n$, then it has the $B$-regular structure up to the generation $t_n$. The detailed construction will be postponed to {\it Subpart 3}.\\

With the help of \eqref{tstar}, we get that
\begin{align}\label{maxPtstar}
&\text{R.H.S. of }\eqref{combound01}\nonumber\\
\leq &  c_{11}^{t_nB^{s_n}}n^{|{\bf t}_{\J}|}\sup_{x_v^*\in {\mbb N}, x_v^*<n, v\in {\bf t}^*_\J} e^{-\lambda\sum_{v\in\mathbf{t}_\J^*} (x_v^*)^\alpha}1_{\l\{\min\limits_{|u|=t_n}\sum\limits_{\rho\prec v\preceq u}x_v^*\geq I_1(n)\r\}}.
\end{align}
  Note that $\min_{|u|=t_n}\sum_{\rho\prec v\preceq u}x_v^*\geq I_1(n)$ leads to $\sum_{|u|=t_n}\sum_{\rho\prec v\preceq u}x_v^*\geq I_1(n)B^{s_n}$, which means
\[
\sum_{i=1}^{k_0} \sum_{|u|=i}x_u^*+\sum_{i=1+k_0}^{t_n}\frac{\sum_{|u|=i}x_u^*}{B^{i-k_0}}\geq I_1(n),
\]
with $k_0=t_n-s_n$.
Note that
\beqnn
\#\{u\in {\bf t}_\J^*: |u|=i\}=\begin{cases}1, & i\leq k_0;\cr
B^{i-k_0}, &k_0<i\leq t_n.
\end{cases}
\eeqnn
Write $\bar{x}_i^*=\sum_{|u|=i}x_u^*$ for $i\leq k_0$ and $\bar{x}^*_i=\frac{\sum_{|u|=i}x_u^*}{B^{i-k_0}}$ for $k_0<i\leq t_n$. So we have
\[
1_{\{\min_{|u|=t_n}\sum_{\rho\prec v\preceq u}x_v^*\geq I_1(n)\}}\leq 1_{\{\sum_{i=1}^{t_n}\bar{x}_i^*\geq I_1(n)\}}.
\]
Consequently,
\begin{equation}\label{cond}
\sup_{x_v^*\in {\mbb N}\cap[0,n), v\in {\bf t}^*_\J} e^{-\lambda\sum_{v\in\mathbf{t}_\J^*} (x_v^*)^\alpha}1_{\l\{\min\limits_{|u|=t_n}\sum\limits_{\rho\prec v\preceq u}x_v^*\geq I_1(n)\r\}}\leq \sup_{x_v^*\in {\mbb N}\cap[0,n), v\in {\bf t}^*_\J} e^{-\lambda\sum_{v\in\mathbf{t}_\J^*} (x_v^*)^\alpha}1_{\l\{\sum\limits_{i=1}^{t_n}\bar{x}_i^*\geq I_1(n)\r\}}.
\end{equation}
Moreover, as $\alpha>1$, by convexity of $x\mapsto x^\alpha$ on $\mR_+$, we obtain that for $k_0<i\leq t_n$,
\[
\sum_{|u|=i}\Big(x_u^*\Big)^\alpha\geq B^{i-k_0} \Big(\frac{\sum_{|u|=i}x_u^*}{B^{i-k_0}}\Big)^\alpha= B^{i-k_0}(\bar{x}_i^*)^\alpha,
\]
and $\sum_{|u|=i}(x_u^*)^\alpha=(\bar{x}_i^*)^\alpha$ for $1\leq i\leq k_0$. Again using convexity implies that, for any $\mu_i>0$,
\begin{align*}
\sum_{v\in\mathbf{t}^*} (x_v^*)^\alpha\geq &\sum_{i=1}^{k_0}(\bar{x}_i^*)^\alpha+\sum_{i=k_0+1}^{t_n} B^{i-k_0}(\bar{x}_i^*)^\alpha\\
=& \sum_{i=1}^{k_0}(\bar{x}_i^*)^\alpha+\sum_{i=1+k_0}^{t_n}B^{i-k_0}\mu_i^{-\alpha}(\mu_i\bar{x}_i^*)^\alpha\\
\geq & \l(k_0+\sum_{i=1+k_0}^{t_n}B^{i-k_0}\mu_i^{-\alpha}\r)
\left(\frac{\sum_{i=1}^{k_0}\bar{x}_i^*
+\sum_{i=1+k_0}^{t_n}B^{i-k_0}\mu_i^{1-\alpha}\bar{x}_i^*}
{k_0+\sum_{i=1+k_0}^{t_n}B^{i-k_0}\mu_i^{-\alpha}}\right)^{\alpha}.
\end{align*}
By taking $\mu_i>0$ such that $B^{i-k_0}\mu_i^{1-\alpha}=1$, one sees that given $\sum_{i=1}^{t_n}\bar{x}_i^*\geq I_1(n)$, we have
\beqlb\label{cond01}
\sum_{v\in\mathbf{t}^*} (x_v^*)^\alpha\ar\geq\ar  \Big(k_0+\sum_{i=1+k_0}^{t_n}B^{-\frac{i-k_0}{\alpha-1}}\Big)^{1-\alpha}
\left(\sum_{i=1}^{t_n}\bar{x}_i^*\right)^\alpha\cr
\ar\geq\ar \frac{\left((I_A(p)-\eta)\sqrt{n}-t_n\right)^{\alpha}}{\l(t_n-s_n+\frac{1}{B^{1/(\alpha-1)}-1}\r)^{\alpha-1}}.
\eeqlb
With \eqref{maxPt}, \eqref{combound01} and \eqref{combound02} in hand,  plugging \eqref{cond01} and \eqref{cond} into \eqref{maxPtstar} yields that
\begin{align}
&\text{L.H.S. of }\eqref{combound}\nonumber\\
&\quad\leq  c_{11}t_n B^{s_n}e^{-\lambda n^\alpha}+ (c_{11}n)^{t_nB^{s_n}} \exp\l\{- \frac{\lambda\left((I_A(p)-\eta)\sqrt{n}-t_n\right)^{\alpha}}{\l(t_n-s_n+\frac{1}{B^{1/(\alpha-1)}-1}\r)^{\alpha-1}}\r\}.
\end{align}
We have completed the proof of \eqref{combound}.

{\it Subpart 3:} We now explain how to construct $\{x_v^*, v\in\mathbf{t}^*\}$.
Recall that $k_0=t_n-s_n$ and  $|Z_k^\mathbf{t}|$ is the number of particles at the $k$-th generation in $\mathbf{t}$. Then
\[
|Z_{i+k_0}^{\mathbf{t}}|\geq B^i, \forall 1\leq i\leq t_n-k_0=s_n.
\]
This shows that there are sufficiently many particles in $\mathbf{t}$. As always, we say that $s_u:=\sum_{\rho<v\leq u}x_v$ is the position of $u\in\mathbf{t}$ and $x_u$ its displacement. For each $u\in \mathbf{t}$ and that $\kappa(u)$ is the number of its children, if it is selected to be in $\mathbf{t}^*$, $x_u^*=x_u$ (but its ancestor line might be changed, so might be its position).

At the $(t_n-1)$-th generation, there are at least $B^{s_n-1}$ particles in $\mathbf{t}$. Let us rearrange them according to their positions: $s_{u_{(1)}}\geq s_{u_{(2)}}\geq\cdots s_{u_{(B^{s_n-1})}}\geq\cdots$. We only take $u_{(j)}, j=1,\cdots, B^{s_n-1}$ (together with their descedants) to be in $\mathbf{t}^*$. Those particles are said to be selected. For $u_{(j)}$, if $\kappa(u_{(j)})=B$, we jump to consider $u_{(j+1)}$. Otherwise, we prune $B-\kappa(u_{(j)})$ unselected particles at the $t_n$-th generation and graft them to $u_{(j)}$. The fact that  $|Z_{t_n}^{\mathbf{t}}|=B^{s_n}$ ensures that $u_{(j)}$ finally has $B$ children for any $\forall 1\leq j\leq B^{s_n-1}$.

Suppose that we have already got the $B^{i+1}$ particles with their descendants for the generation $k_0+i+1$ of $\mathbf{t}^*$. We are interested in their parents at the $(k_0+i)$-th generation in $\mathbf{t}$, the number of which is at least $B^{i}$. We rearrange them according to their positions: $s_{w_{(1)}}\geq s_{w_{(2)}}\geq\cdots s_{w_{(B^{i})}}\geq\cdots$.
We select $w_{(j)},\ j=1,\cdots, B^i$ with their descendants in $\mathbf{t}^*$,  For $w_{(j)}$, if $\kappa(w_{(j)})=B$, we jump to consider $u_{(j+1)}$. Otherwise, we prune $B-\kappa(w_{(j)})$ unselected particles at the $(k_0+i+1)$-th generation together with their descendants and graft these subtrees to $w_{(j)}$.

We continue this backward construction until the $k_0$-th generation where there is only one particle in $\mathbf{t}^*$. We then take directly all its ancestors in $\mathbf{t}$ to establish the single branch of $\mathbf{t}^*$.

In the following figure, we give an example of this construction. The horizontal axis represents generations and the vertical axis represents positions of particles. In FigA, from the tree ${\bf t}$, we choose $B^{s_n}$ particles in $\J$ and colour them in red. In FigB, we subtract all the ancestors of the red particles and remove the others. In FigC, we do the pruning and grafting for the last two generations as explained above. In FigD, we get the final labelled tree ${\bf t^*}$ in blue while the red particles are those chosen in $\J$. The blue dashed lines link the particles and their added descendants at each step.

\includegraphics[width=16cm]{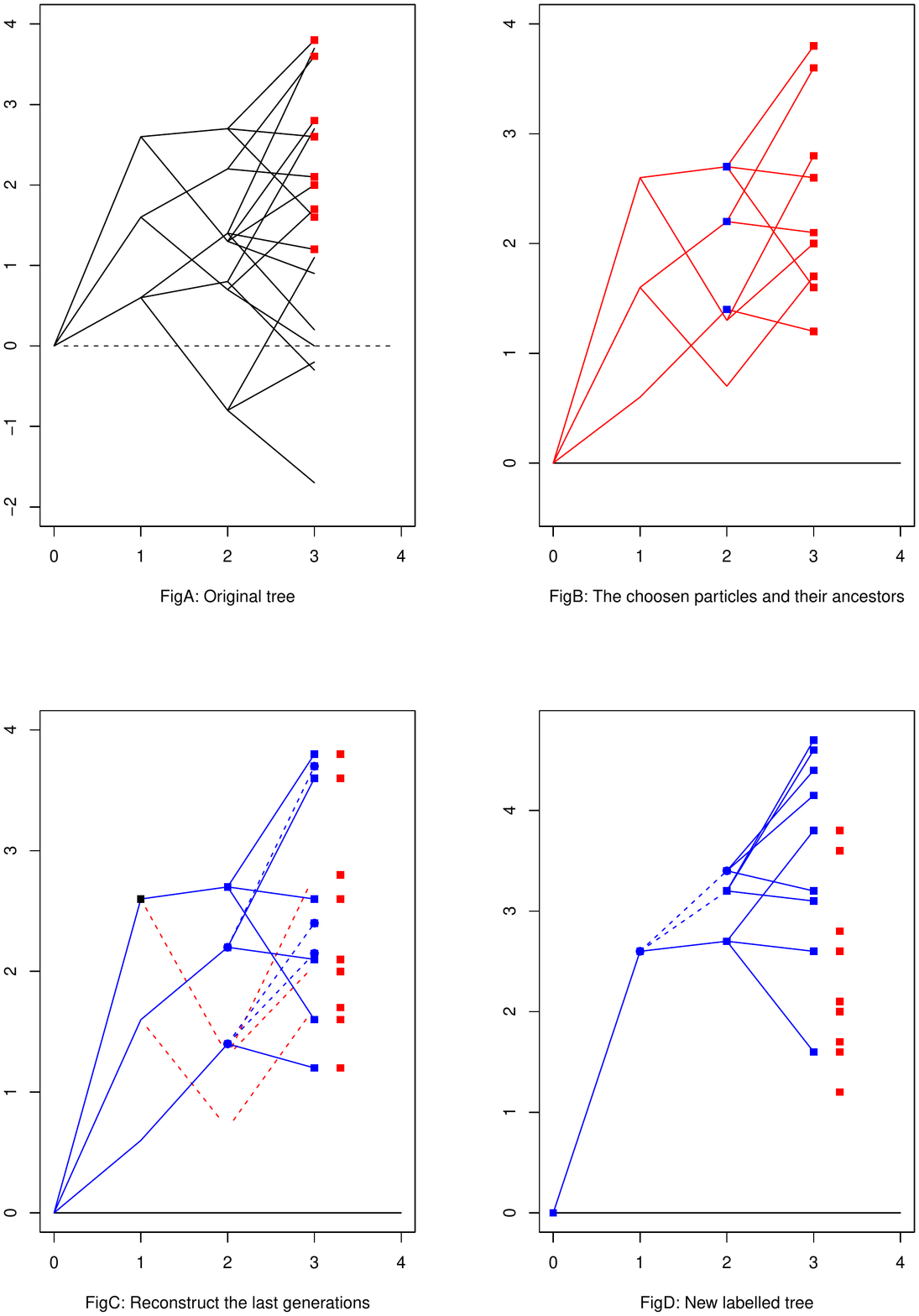}\label{Fig1}

Observe that only a part of the particles in $\mathbf{t}$ has been selected and that the positions at the final generation of ${\bf t^*}$ are all getting higher. \eqref{tstar} is satisfied.

\subsubsection{LDP when the tree is regular}

Unfortunately, the arguments above do not all work when the embedding tree $\cal T$ is regular. We could only get the upper and lower bound for the large deviation behaviours:
\begin{multline}\label{Main0303}
-C_\alpha\leq\liminf_{n\rightarrow\infty}\frac{1}{n^{\alpha/2}}\log\P\left(\bar{Z}_n(\sqrt{n}A)-\nu(A)\geq \Delta\right) \\
\leq \limsup_{n\rightarrow\infty}\frac{1}{n^{\alpha/2}}\log\P\left(\bar{Z}_n(\sqrt{n}A)-\nu(A)\geq \Delta\right) \leq -c_\alpha.
\end{multline}

\paragraph{Lower bound} No matter whether $\alpha<1$ or $\alpha\geq1$, we only consider the first generation and suppose that
\[
S_u\in [(x-\eta)\sqrt{n},(x+\eta)\sqrt{n}], \quad\forall |u|=1,
\]
where $|x|=I_A(p)+\ez$ with $\inf_{y\in[x-\eta,x+\eta]}\nu(A-y)\geq p+\delta$ for some $\delta,\eta,\ez>0$. Given $\mathbf{E}:=\{S_u\in [(x-\eta)\sqrt{n},(x+\eta)\sqrt{n}], \forall |u|=1\}$,
\[
\liminf_{n\rightarrow\infty} \bar{Z}_n(\sqrt{n}A)\geq p+\delta.
\]
Then
\begin{align*}
\P\left(\bar{Z}_n(\sqrt{n}A)-\nu(A)\geq \Delta\right)\geq & \P(S_u\in [(x-\eta)\sqrt{n},(x+\eta)\sqrt{n}], \forall |u|=1; \bar{Z}_n(\sqrt{n}A)\geq p)\\
\geq &\P(\bar{Z}_n(\sqrt{n}A)\geq p\vert\mathbf{E})\P(\mathbf{E}),
\end{align*}
where
\[
\liminf_{n\rightarrow\infty}\P(\bar{Z}_n(\sqrt{n}A\geq p\vert\mathbf{E}) \geq \P(\liminf_{n\rightarrow\infty} \bar{Z}_n(\sqrt{n}A)\geq p\vert E)=1.
\]
On the other hand,
\[
\P(\mathbf{E})=\P(X\in  [(x-\eta)\sqrt{n},(x+\eta)\sqrt{n}])^b=(1+o_n(1))e^{-\lambda b (x-\eta)^\alpha n^{\alpha/2}}.
\]
We thus deduce that
\[
\liminf_{n\rightarrow\infty}\frac{1}{n^{\alpha/2}}\log\P\left(\bar{Z}_n(\sqrt{n}A)-\nu(A)\geq \Delta\right) \geq - \lambda b (x-\eta)^\alpha.
\]
We obtain the lower bound in \eqref{Main0303} with $C_\alpha=\lambda b (I_A(p))^\alpha$.

\paragraph{Upper bound}

When $\alpha\leq1$, the arguments in Section 4.1.1 still work for regular tree. So the upper bound is obtained with $c_\alpha=\lambda(I_A(p))^\alpha$ in the case of $\alpha\leq 1$.

When $\alpha>1$, \eqref{Main03up+} still holds. So, for $B_n=[-(I_A(p)-\eta)\sqrt{n},(I_A(p)-\eta)\sqrt{n}]$ and $t_n=\l\lfloor\frac{\alpha\log b}{2}\log n\r\rfloor$,
\begin{align*}
\P(\bar{Z}_n(\sqrt{n}A)-\nu(A)\geq \Delta)\leq \P(Z_{t_n}(B_n^c)\geq \delta b^{t_n}/2)+C_1 e^{-c_9 b^{t_n}}.
\end{align*}
Observe that
\[
\frac{Z_{t_n}(B_n^c)}{b^{t_n}}=\sum_{|u|=t_n}\frac{1_{|S_u|\geq (I_A(p)-\eta)\sqrt{n}}}{b^{t_n}}\leq \frac{1}{(I_A(p)-\eta)\sqrt{n}}\sum_{|u|=t_n}\frac{|S_u|}{b^{t_n}}.
\]
Moreover, $|S_u|\leq \sum_{\rho\prec v\preceq u}|X_u|$ implies that
\[
\frac{Z_{t_n}(B_n^c)}{b^{t_n}}\leq \frac{1}{(I_A(p)-\eta)\sqrt{n}}\sum_{|u|=t_n}\frac{|S_u|}{b^{t_n}}\leq \frac{1}{(I_A(p)-\eta)\sqrt{n}}\sum_{k=1}^{t_n}\sum_{|u|=k}\frac{|X_u|}{b^k}.
\]
It follows that
\begin{align*}
\P(\bar{Z}_n(\sqrt{n}A)-\nu(A)\geq \Delta)\leq & \P\left(\sum_{k=1}^{t_n}\sum_{|u|=k}\frac{|X_u|}{b^k}\geq \delta(I_A(p)-\eta)\sqrt{n}\right)+C_1 e^{-c_9 n^{\alpha/2}}\\
=& \P\left(\left(\sum_{k=1}^{t_n}\sum_{|u|=k}\frac{|X_u|}{b^k}\right)^\alpha \geq [\delta(I_A(p)-\eta)\sqrt{n}]^\alpha\right)+C_1 e^{-c_9 n^{\alpha/2}}.
\end{align*}
By convexity of $x\mapsto x^\alpha$, for any $\mu_k>0$, one sees that
\begin{align*}
\left(\sum_{k=1}^{t_n}\sum_{|u|=k}\frac{|X_u|}{b^k}\right)^\alpha =&\left(\sum_{k=1}^{t_n}\sum_{|u|=k}\frac{\mu_k}{b^k}(\mu_k^{-1}|X_u|)\right)^\alpha\\
\leq & (\sum_{k=1}^{t_n}\mu_k)^\alpha\frac{\sum_{k=1}^{t_n}\sum_{|u|=k}\mu_k b^{-k}(\mu_k^{-1}|X_u|)^\alpha}{\sum_{k=1}^{t_n}\mu_k}\\
\leq& \mu^{\alpha-1}\sum_{k=1}^{t_n}\sum_{|u|=k}b^{-k}\mu_k^{1-\alpha}|X_u|^\alpha,
\end{align*}
where $\mu=\sum_{k=1}^{\infty}\mu_k$. We are going to take a decreasing sequence $\mu_k=k^{-2}$ so that $\mu<\infty$. Therefore, for any $\theta>0$,
\begin{align}\label{Main0303up}
&\P(\bar{Z}_n(\sqrt{n}A)-\nu(A)\geq \Delta)\nonumber\\
\leq & \P\left( \mu^{\alpha-1}\sum_{k=1}^{t_n}\sum_{|u|=k}b^{-k}\mu_k^{1-\alpha}|X_u|^\alpha\geq [\delta(I_A(p)-\eta)\sqrt{n}]^\alpha\right)+C_1 e^{-c_9 n^{\alpha/2}}\nonumber\\
\leq & e^{-\frac{\theta}{\mu^{\alpha-1}}\delta^\alpha (I_A(p)-\eta)^\alpha n^{\alpha/2}}\E\left[\exp\{\theta\sum_{k=1}^{t_n}\sum_{|u|=k}b^{-k}\mu_k^{1-\alpha}|X_u|^\alpha\}\right]+C_1 e^{-c_9 n^{\alpha/2}}.
\end{align}
We then show that for $\theta>0$ sufficiently small such that $\sup_{k\geq 1}\theta b^{-k}k^{2(\alpha-1)} \leq \lambda/2$, there exists $c_\lambda>0$ such that
\[
\E\left[\exp\{\theta\sum_{k=1}^{t_n}\sum_{|u|=k}b^{-k}\mu_k^{1-\alpha}|X_u|^\alpha\}\right]\leq e^{c_{\lambda}\theta t_n^{2\alpha-1}}.
\]
In fact, by independence,
\begin{align*}
\E\left[\exp\{\theta\sum_{k=1}^{t_n}\sum_{|u|=k}b^{-k}\mu_k^{1-\alpha}|X_u|^\alpha\}\right]=&\prod_{k=1}^{t_n}\prod_{|u|=k}\E\left[\exp\{\theta b^{-k}k^{2(\alpha-1)}|X_u|^\alpha\}\right]\\
=& \prod_{k=1}^{t_n}\prod_{|u|=k}\left(1+\int_0^\infty \theta b^{-k}k^{2(\alpha-1)}e^{\theta b^{-k}k^{2(\alpha-1)}x}\P(|X_u|^\alpha\geq x)dx\right).
\end{align*}
The tail distribution of $X$ shows that $\P(|X|\geq x)\leq c_{11} e^{-\lambda x^\alpha}$ for any $x\geq0$. It follows that
\begin{align*}
\E\left[\exp\{\theta\sum_{k=1}^{t_n}\sum_{|u|=k}b^{-k}k^{2(\alpha-1)}|X_u|^\alpha\}\right]\leq &\prod_{k=1}^{t_n}\prod_{|u|=k}\left(1+c_{11}\int_0^\infty \theta b^{-k}k^{2(\alpha-1)}e^{\theta b^{-k}k^{2(\alpha-1)}x-\lambda x}dx\right)\\
=&\prod_{k=1}^{t_n}\prod_{|u|=k}\left(1+c_{11}\frac{\theta b^{-k}k^{2(\alpha-1)}}{\lambda-\theta b^{-k}k^{2(\alpha-1)}}\right)\\
\leq & \exp\left\{\sum_{k=1}^{t_n}\sum_{|u|=k} c_{11} \frac{\theta b^{-k}k^{2(\alpha-1)}}{\lambda-\theta b^{-k}k^{2(\alpha-1)}}\right\}<e^{c_{\lambda}\theta t_n^{2\alpha-1}}.
\end{align*}
Plugging it into \eqref{Main0303up} yields that
\[
\P(\bar{Z}_n(\sqrt{n}A)-\nu(A)\geq \Delta)\leq e^{-\frac{\theta}{\mu^{\alpha-1}}\delta^\alpha (I_A(p)-\eta)^\alpha n^{\alpha/2}+c_\lambda \theta t_n^{2\alpha-1}}+C_1 e^{-c_9 n^{\alpha/2}}.
\]
Recall that $t_n=O(\log n)$. We have
\[
\limsup_{n\rightarrow\infty}\frac{1}{n^{\alpha/2}}\log \P(\bar{Z}_n(\sqrt{n}A)-\nu(A)\geq \Delta)\leq -c_\alpha.
\]

\subsection{Proof of Theorem \ref{Main04}: step size has Gumbel tail distribution}

In this section, we assume that the tail distribution of step size is of Gumbel's type, in other words, $\P(X\geq x)=\Theta(1) e^{-e^{x^\alpha}}$ with $\alpha>0$, for $x\rightarrow\infty$. In what follows, we are devoted to demonstrating that if $I_A$ is finite and continuous at $p=\nu(A)+\Delta$, then
\begin{equation}\label{Main04cvg}
\lim_{n\rightarrow\infty}{n^{-\frac{\alpha}{2(\alpha+1)}}}\log\left[-\log
\P\left(\bar{Z}_n(\sqrt{n}A)-\nu(A)>\Delta\right)\right]
= \left(y_{\alpha}I_A(p)\log b\right)^{\frac{\alpha}{\alpha+1}},
\end{equation}
where $y_\alpha=\frac{(1+\alpha)\log B}{(1+\alpha)\log B-\log b} $.

The ideas of proof are similar to that used in Section \ref{proofsupex}. However, we do not need to assume $B>b$.

\subsubsection{Lower bound of \eqref{Main04cvg}} As stated in Section \ref{proofsupex}, for any sufficiently small $\delta>0$, there exist $x_0\in\mR$, $\ez,\eta>0$ such that
\[
\inf_{|y-x_0|\leq \eta}\nu(A-x)\geq p+\delta,\ I_A(p)< x_0-\eta<x_0 \leq I_A(p)+\ez.
\]
Take $d\geq b$ an integer such that $p_d>0$. (Again, if $B<\infty$,  then we choose $d=B$. If $B=\infty$, we will let $d\rightarrow\infty$ later.) Let
\beqnn
t_n=\left\lfloor tn^{\frac{\alpha}{2(\alpha+1)}}-\frac{2\log n}{\log b}\right\rfloor, s_n=\left\lfloor\frac{\log n+t_n\log b}{\log d}\right\rfloor\wedge t_n, y_\alpha(d)=\frac{(1+\alpha)\log d}{(1+\alpha)\log d-\log b}
\eeqnn
where $t>0$ will be determined later on. With these $d$, $t_n$, $s_n$, recall event $\mathcal{E}_{t_n,d,b}$  from \eqref{EVENTE}.  It follows that
\begin{align}\label{Main04lower}
&\P(\bar{Z}_n(\sqrt{n}A)-\nu(A)\geq \Delta)\geq  \P(\bar{Z}_n(\sqrt{n}A)-\nu(A)\geq \Delta; \mathcal{E}_{t_n,d,b}).
\end{align}
Define
\[
\mathcal{M}_2=\l\{\zeta\in {\cal M}: |\zeta|=d^{s_n}+b^{t_n}-b^{s_n}, \zeta( [(x_0-\eta)\sqrt{n}, (x_0+\eta)\sqrt{n}])\geq d^{s_n}\r\}.
\]
  We  have for any $\zeta\in {\cal M}_2$,
\[
\frac{1}{|\zeta|}\sum_{x\in\zeta}\nu(A-\frac{x}{\sqrt{n}})\geq \frac{d^{s_n}}{b^{t_n}-b^{s_n}+d^{s_n}}\inf_{y\in[x_0-\eta, x_0+\eta]}\nu(A-y)\geq \frac{n}{n+1}(p+\delta),
\]
which, together with (5) of Lemmma \ref{LP}, implies that for all $n$ sufficiently large,
\[
\frac{1}{|\zeta|}\sum_{x\in\zeta}\nu_{n-t_n}(\sqrt{n}A-x)\geq p+\delta/2.
\]
Note that, $Z_{t_n}\in{\cal M}_2$, given $\mathcal{E}_{t_n,d,b}$. Consequently,
\begin{align*}
&\P(\bar{Z}_n(\sqrt{n}A)\geq p; \mathcal{E}_{t_n,d,b})\\&\quad\geq  \sum_{\zeta\in{\cal M}_2}\P(Z_{t_n}=\zeta; \mathcal{E}_{t_n,d,b})\P\left(\bar{Z}^\zeta_{n-t_n}(\sqrt{n}A)\geq p\right)\\
&\quad\geq  \sum_{\zeta\in{\cal M}_2}\P(Z_{t_n}=\zeta; \mathcal{E}_{t_n,d,b})\P\left(\bar{Z}^\zeta_{n-t_n}(\sqrt{n}A)\geq \frac{1}{|\zeta|}\sum_{x\in\zeta}\nu_{n-t_n}(\sqrt{n}A-x)-\delta/2\right)\\
&\quad\geq  \P(\mathcal{E}_{t_n,d,b})(1- C_1e^{-C_2 |\zeta| \delta^2/4}),
\end{align*}
where the last inequality follows from Lemma \ref{LP+}. As $|\zeta|\geq b^{t_n}$, for $n$ large enough,
\begin{equation}\label{Main04lower+}
\P(\bar{Z}_n(\sqrt{n}A)-\nu(A)\geq \Delta)\geq \frac{1}{2}\P(\mathcal{E}_{t_n,d,b}).
\end{equation}
Recall that $u^*=(1,\cdots,1)$ with $|u^*|=t_n-s_n$. It remains to consider $\P(\mathcal{E}_{t_n,d,b})$, which by \eqref{keyevent} is
\begin{equation}\label{Main04lowerbd}
\P(\mathcal{E}_{t_n,d,b})\geq p_b^{b^{t_n}}p_d^{d^{s_n}}\P^{\bf t}(S_u\in[(x_0-\eta)\sqrt{n},(x_0+\eta)\sqrt{n}], \forall u>u^*, |u|=t_n).
\end{equation}
Let $\cx=(x_0y_\alpha(d) \log b)^{\frac{1}{\alpha+1}}$ and $t=\frac{\cx^\alpha}{\log b}$. We suppose that for any ancestor of $u^*$: $u\preceq u^*$,
\[
X_u\in \l[ \cx n^{\frac{1}{2(\alpha+1)}},\,  \cx n^{\frac{1}{2(\alpha+1)}}+\frac{\eta}{2t}n^{\frac{1}{2(\alpha+1)}}\r],
\]
and that for any descendant of $u^*$: $u^*\prec u$ s.t. $|u|\leq t_n$,
\[
X_u\in\left[\left(\cx^\alpha n^{\frac{\alpha}{2(\alpha+1)}}-(|u|+s_n-t_n)\log d\right)^{\frac{1}{\alpha}}, \, \left(\cx^\alpha n^{\frac{\alpha}{2(\alpha+1)}}-(|u|+s_n-t_n)\log d\right)^{\frac{1}{\alpha}}+\frac{\eta}{2t}n^{\frac{1}{2(\alpha+1)}}\right].
\]
This ensures that for any descendant $u$ of $u^*$ at the $t_n$-th generation, its position $S_u$ satisfies
\[
\Bigg\vert S_u-\left[(t_n-s_n)\cx n^{\frac{1}{2(\alpha+1)}}+\sum_{k=1}^{s_n}\l(\cx^\alpha-{k}{n^{-\frac{\alpha}{2(\alpha+1)}}}\log d\r)^{\frac{1}{\alpha}}n^{\frac{1}{2(\alpha+1)}}\right]\Bigg\vert\leq \eta \sqrt{n}/2,
\]
where as $n\rar\infty$, $t_n\approx tn^{\frac{\alpha}{2(\alpha+1)}}$, $s_n\approx \frac{\log b}{\log d}t_n$ and
\begin{align*}
&(t_n-s_n)\cx n^{\frac{1}{2(\alpha+1)}}+\sum_{k=1}^{s_n}\l(\cx^\alpha-{k}{n^{-\frac{\alpha}{2(\alpha+1)}}}\log d\r)^{\frac{1}{\alpha}}n^{\frac{1}{2(\alpha+1)}}\\
&\quad=\sqrt{n}\left[\l(1-\frac{\log b}{\log d}\r)t\cx+\int_0^{\frac{\log b}{\log d}t}(\cx^\alpha-x\log d)^{\frac{1}{\alpha}}dx+o_n(1)\right]\\
&\quad= \sqrt{n}( x_0+o_n(1)).
\end{align*}
So for all $n$ large enough, we have
\[
S_u\in[(x_0-\eta)\sqrt{n},(x_0+\eta)\sqrt{n}], \quad \forall u^*\prec u,\,  |u|=t_n.
\]
As a result,
\begin{align*}
&\P^{\bf t}(S_u\in[(x_0-\eta)\sqrt{n},(x_0+\eta)\sqrt{n}], \forall u^*\prec u, |u|=t_n)\\
\geq & \P\left(X\in [ \cx n^{\frac{1}{2(\alpha+1)}}, \cx n^{\frac{1}{2(\alpha+1)}}+\frac{\eta}{2t}n^{\frac{1}{2(\alpha+1)}}]\right)^{t_n-s_n}\\
&\times \prod_{k=1}^{s_n}\P\left(X\in\left[\left(\cx^\alpha n^{\frac{\alpha}{2(\alpha+1)}}-k\log d\right)^{\frac{1}{\alpha}}, \left(\cx^\alpha n^{\frac{\alpha}{2(\alpha+1)}}-k\log d\right)^{\frac{1}{\alpha}}+\frac{\eta}{2t}n^{\frac{1}{2(\alpha+1)}}\right]\right)^{d^{k}}\\
\geq &\l(\frac{1}{2}\r)^{d^{s_n+1}}c_{12}^{t_n-s_n}\exp\l\{-(t_n-s_n)e^{\cx^\alpha n^{\frac{\alpha}{2(\alpha+1)}}}-\sum_{k=1}^{s_n}d^k e^{\cx^\alpha n^{\frac{\alpha}{2(\alpha+1)}}-k\log d}\r\}.
\end{align*}
Plugging it into \eqref{Main04lowerbd} implies that
\begin{align*}
&\P(\mathcal{E}_{t_n,d,b})\geq p_b^{b^{t_n}}p_d^{nb^{t_n}}\l(\frac{1}{2}\r)^{d b^{t_n}}c_{12}^{t_n}\exp\l\{-t_ne^{\cx^\alpha n^{\frac{\alpha}{2(\alpha+1)}}}\r\}
= \exp\l\{-(1+o_n(1))t_ne^{\cx^\alpha n^{\frac{\alpha}{2(\alpha+1)}}}\r\}.
\end{align*}
Going back to \eqref{Main04lower+}, we hence conclude the lower bound by letting $d\uparrow B$.
\subsubsection{Upper bound}

Let $B_n:=[(-I_A(p)+\eta)\sqrt{n}, (I_A(p)-\eta)\sqrt{n}]$, $t_n=\left\lfloor tn^{\frac{\alpha}{2(\alpha+1)}}\right\rfloor$ with some $t>0$. 
Recall \eqref{nutp}. One sees that for any $\zeta\in\mathcal{M}$ and $n$ large enough,
\[
\l\{\frac{\zeta(B_n^c)}{|\zeta|}\leq \delta/2\r\}\subset\l\{\frac{1}{|\zeta|}\sum_{x\in\zeta}\nu_{n-t_n}(\sqrt{n}A-x)\leq p-\delta/4\r\}.
\]
Again, similar to \eqref{Main03up+}, we thus have
\begin{equation}\label{Main04up}
\P(\bar{Z}_n(\sqrt{n}A)-\nu(A)\geq \Delta)\leq \P(Z_{t_n}(B_n^c)\geq \delta b^{t_n}/2)+C_1 e^{-c_9 b^{t_n}}.
\end{equation}
Here we need to treat $\P(Z_{t_n}(B_n^c)\geq \delta b^{t_n}/2)$ separately in the two following cases.

\paragraph{Upper bound: $B=\infty$} In this case, $y_\alpha=1$. By Markov inequality, we have
\begin{align*}
\P(Z_{t_n}(B_n^c)\geq \delta b^{t_n}/2)\leq & \frac{2}{\delta} b^{-t_n}\E[Z_{t_n}(B_n^c)]\\
= & \frac{4}{\delta} b^{-t_n}m^{t_n}\nu_{t_n}((I_A(p)-\eta)\sqrt{n},\infty).
\end{align*}
 In view of (2) of Lemma \ref{A2}, one has immediately,
\[
\liminf_{n\rightarrow\infty}\frac{1}{n^{\frac{\alpha}{2(\alpha+1)}}}\log\left[-\log \P(Z_{t_n}(B_n^c)\geq \delta b^{t_n}/2)\right]\geq \l(\frac{I_A(p)-\eta}{t}\r)^\alpha.
\]
In \eqref{Main04up}, $b^{t_n}=\exp\{ t \log b n^{\frac{\alpha}{2(\alpha+1)}}\}$. We choose $t=\frac{(I_A(p)-\eta)^{\frac{\alpha}{\alpha+1}}}{(\log b)^{\frac{1}{\alpha+1}}}$ so that $t\log b=\l(\frac{I_A(p)-\eta}{t}\r)^\alpha$. We thus conclude that
\[
\liminf_{n\rightarrow\infty}\frac{1}{n^{\frac{\alpha}{2(\alpha+1)}}}\log\left[-\log \P(\bar{Z}_n(\sqrt{n}A)-\nu(A)\geq \Delta)\right]\geq \left((I_A(p)-\eta)y_\alpha\log b\right)^{\frac{\alpha}{2(\alpha+1))}}.
\]
\paragraph{Upper bound: $b\leq B<\infty$} Following the same idea applied in \eqref{Main03upkey}, we take
\[
s_n=\l\lfloor \frac{t_n\log b-\log n}{\log B}\r\rfloor, \ k_0=t_n-s_n,
\]
and reconstruct the special subtree $\mathbf{t}^*_\J$. It follows from the Gumbel tail distribution of $X$ that there exists $c_{13}\geq 1$ such that
\[
\P(|X|\geq x)\leq c_{13} e^{-e^{x^\alpha}},\quad\text{and }  \forall x\geq0.
\]
Similarly to \eqref{Main03upkey}, \eqref{maxPt}, \eqref{combound01} and \eqref{combound02}, one sees that
%
%
\begin{multline}\label{Main04key}
\P(Z_{t_n}(B_n^c)\geq \delta b^{t_n}/2) \\
\leq 2 B^{t_n B^{s_n}}\left[ c_{13}t_n B^{s_n}e^{-e^{n^{\alpha}}}+ c_{13}^{t_nB^{s_n}}n^{t_n B^{s_n}}\max_{x_u^*\geq0, u\in\mathbf{t}^*_\J\atop \sum_{k=1}^{t_n}\bar{x}_k^*\geq I_1(n)} \exp\left\{-\sum_{u\in \mathbf{t}^*_\J}e^{(x_u^*)^\alpha}\right\}\right].
\end{multline}
Note that
\[
\sum_{u\in \mathbf{t}^*_\J}e^{(x_u^*)^\alpha}\geq \sum_{k=1}^{t_n-s_n}\sum_{|u|=k}e^{(x_u^*)^\alpha}+\sum_{k=1}^{s_n}\sum_{|u|=k+t_n-s_n}e^{(x_u^*\vee M)^\alpha}-\sum_{k=1}^{s_n}B^k e^{M^\alpha}.
\]
where $M\geq0$ is chosen so that $x\mapsto e^{x^\alpha}$ is convex on $[M,\infty)$. Immediately, $\sum_{|u|=k+t_n-s_n}e^{(x_u^*\vee M)^\alpha}\geq B^k e^{(\frac{\sum_{}x_u^*\vee M}{B^k})^\alpha}$. So,
\[
\sum_{u\in \mathbf{t}^*_\J}e^{(x_u^*)^\alpha}\geq \sum_{k=1}^{t_n-s_n} e^{(\bar{x}_k^*)^\alpha}+\sum_{k=1}^{s_n} B^k e^{(\bar{x}_{k+t_n-s_n}^*)^\alpha}-\sum_{k=1}^{s_n}B^k e^{M^\alpha}.
\]
where $\bar{x}^*_k=\frac{\sum_{|u|=k}x_u^*}{B^{(k-t_n+s_n)_+}}$. Let
\(
\Xi_n:=\max\{(\bar{x}_k^*)^\alpha+(k-t_n+s_n)_+\log B; 1\leq k\leq t_n\}.
\)
Then,
\[
\sum_{u\in \mathbf{t}^*_\J}e^{(x_u^*)^\alpha}\geq e^{\Xi_n}- B^{s_n+1}e^{M^\alpha}.
\]
Plugging it into \eqref{Main04key} implies that
\begin{equation}\label{keyup04}
\P(Z_{t_n}(B_n^c)\geq \delta b^{t_n}/2) \leq e^{-(1+o_n(1))n^{\alpha}}+2(c_{13}Bn)^{t_n B^{s_n}}\max_{x_u^*\geq0, u\in\mathbf{t}^*_\J\atop \sum_{k=1}^{t_n}\bar{x}_i^*\geq I(n)}\exp\{-e^{\Xi_n}+B^{s_n+1}e^{M^\alpha}\}.
\end{equation}
Note that  $\{\sum_{k=1}^{t_n}\bar{x}_k^*\geq I_1(n)\}$ implies
\[
I_1(n)\leq \sum_{k=1}^{t_n}\bar{x}_k^*\leq (t_n-s_n)\Xi_n^{1/\alpha}+\sum_{k=1}^{s_n}\Big(\Xi_n- k\log B\Big)^{1/\alpha},
\]
and $\Xi_n\geq s_n\log B=( t\log b+o_n(1)) n^{\frac{\alpha}{2(\alpha+1)}}$. Let $K_n:= \Xi_n n^{-\frac{\alpha}{2(\alpha+1)}}$. Then,
\begin{align*}
(I_A(p)-\eta)+o_n(1) \leq & \l(1-\frac{\log b}{\log B}+o_n(1)\r)t K_n^{1/\alpha}+\int_0^{t\log b/\log B}\Big(K_n- x\log B\Big)_+^{1/\alpha}dx\\
\leq& \l(1-\frac{\log b}{\log B}+o_n(1)\r)t K_n^{1/\alpha}+\frac{\alpha K_n^{1+1/\alpha}}{(1+\alpha)\log B}.
\end{align*}
In view of the right hand side of \eqref{keyup04} where $t_n B^{s_n}\leq \frac{t_n}{n}b^{t_n}$, we take $t$ such that $t_n\log b\leq \Xi_n$. In other words, $t\leq K_n/\log b$. This choice entails that
\[
K_n^{1+1/\alpha}\l(\frac{\log B-\log b}{\log B\log b}+o_n(1)+\frac{\alpha}{(1+\alpha)\log B}\r)\geq I_A(p)-\eta+o_n(1).
\]
Thus, setting $\cx= [(I_A(p)-\eta)y_\alpha\log b]^{\frac{\alpha}{\alpha+1}}$ and $t=\frac{\cx}{\log b}$, we have
\[
K_n\geq \cx(1+o_n(1)),\ b^{t_n}\leq e^{\cx n^{\frac{\alpha}{2(\alpha+1)}}}\textrm{ and } t_nB^{s_n}\log n\ll e^{\cx n^{\frac{\alpha}{2(\alpha+1)}}}.
\]
As $\Xi_n=K_n n^{\frac{\alpha}{2(\alpha+1)}}$, applying this to \eqref{keyup04} yields that for sufficiently large $n$,
\begin{align*}
\P(Z_{t_n}(B_n^c)\geq \delta b^{t_n}/2) \leq& e^{-(1+o_n(1))n^{\alpha}}+2(c_{13}Bn)^{t_n B^{s_n}}  \exp\l\{-e^{\cx(1+o_n(1))n^{\frac{\alpha}{2(\alpha+1)}}}+c_{14}t_nB^{s_n}\r\}\\
\leq & e^{-(1+o_n(1))e^{n^{\alpha}}} + \exp\l\{-(1+o_n(1))e^{[(I_A(p)-\eta)y_\alpha\log b]^{\frac{\alpha}{\alpha+1}}n^{\frac{\alpha}{2(\alpha+1)}}}\r\}.
\end{align*}
Therefore,
\[
\liminf\frac{1}{n^{\frac{\alpha}{2(\alpha+1)}}}\log[-\log \P(Z_{t_n}(B_n^c)\geq \delta b^{t_n}/2) ]\geq  [(I_A(p)-\eta)y_\alpha\log b]^{\frac{\alpha}{\alpha+1}}.
\]
Going back to \eqref{Main04up}, as $b^{t_n}=e^{\cx n^{\frac{\alpha}{2(\alpha+1)}}}$, we conclude that
\[
\liminf\frac{1}{n^{\frac{\alpha}{2(\alpha+1)}}}\log[-\log \P(\bar{Z}_n(\sqrt{n}A)-\nu(A)\geq \Delta)]\geq [(I_A(p)-\eta)y_\alpha\log b]^{\frac{\alpha}{\alpha+1}},
\]
which ends the proof of the upper bound of Theorem \ref{Main04} by letting $\eta\downarrow 0$. \qed

\bigskip

{\bf Acknowledgement.} We are grateful to Prof. Zhan Shi for enlightenment on Lemma \ref{iteration}. Hui He is supported by NSFC (No. 11671041, 11531001, 11371061).

\appendix{}
\renewcommand{\theequation}{\Alph{section}.\arabic{equation}}
\section{Appendix}

\begin{lem}\label{A1}If $L_0< \bar{\Lambda}(p_1)$ and $L_{k+1}=F(L_k)=\alpha \inf_{u\in\mR}(-\log p_1+\gamma(u)-uL_k)+L_k$ for any $k\geq0$, then the sequence $(L_k)_{k\geq0}$ is non-decreasing and
\[
\lim_{k\rightarrow\infty}L_k=\bar{\Lambda}(p_1)=\inf_{a>0}\frac{\gamma(a)-\log p_1}{a}.
\]
\end{lem}
 \begin{proof}This sequence $(L_k)_{k\geq1}$ is non-decreasing, so its limit exists.  In fact, if $L\in(0, \Lambda^{-1}(\log\frac{1}{p_1}))$, we have $F(L)\geq L$, in other words,
\[
\inf_{u\in\mR}\{\gamma(u/a)-L u/a\}+\log\frac{1}{p_1}=\inf_{u\in\mR}\{\gamma(u)-L u\}+\log\frac{1}{p_1}\geq0.
\]
Note that the rate function $\gamma$ is convex, so we only need that $\gamma'(u)=L$ and the infimum is $\log\frac{1}{p_1}-\Lambda(L)\geq0$.

Note that if $L_k\leq \theta:=\bar{\Lambda}(p_1)$, $\log\frac{1}{p_1}=\Lambda(\theta)$, and
\[
L_{k+1}=\alpha(\Lambda(\theta)-\Lambda(L_k))+L_k\leq \alpha\Lambda'(\theta)(\theta-L_k)+L_k\leq \theta.
\]
Hence, $\lim_{k\rightarrow\infty}L_k=L_\infty\leq \theta$. Here $\alpha\Lambda'(\theta)\leq 1$ because we need to take $\alpha>0$ small so that
\[
\alpha\gamma(1/\alpha)-\alpha\log m\geq \inf_{a>0}\frac{\gamma(a)-\log p_1}{a}=\theta.
\]
In fact, either $1/\alpha\geq \sup_{t\in\mR_+}\Lambda'(t)\geq \Lambda'(\theta)$, or there exists $t_{1/\alpha}>0$ such that $1/\alpha=\Lambda'(t_{1/\alpha})$ and then $\gamma(1/\alpha)=\frac{t_{t/\alpha}}{\alpha}-\Lambda(t_{1/\alpha})$. This follows that
\[
\alpha\gamma(1/\alpha)-\alpha\log m =t_{1/\alpha}-\alpha(\Lambda(t_{1/\alpha})+\log m)\geq \theta.
\]
As $\Lambda(t)\geq0$, we have $t_{1/\alpha}>\theta$. Consequently, $\Lambda'(\theta)\leq \Lambda'(t_{1/\alpha})=1/\alpha$.

On the other hand, as $L_0\leq L_k\leq \theta$,
\[
L_{k+1}-L_k=\alpha(\Lambda(\theta)-\Lambda(L_k))\geq \alpha\Lambda'(L_k)(\theta-L_k)\geq \alpha\Lambda'(L_0)(\theta-L_k)\geq0.
\]
This shows that $\theta-L_k\rightarrow 0$.
\end{proof}
The following lemma concerns large deviation probabilities of sums of independent random variables. The results are possibly  well-known to some experts or implicitly contained in some articles.
\begin{lem}\label{A2}
Suppose that $\{X_i\}_{i\geq1}$ is a sequence of i.i.d. random variables, having the same distribution as $X$. $X$ is symmetric.
\begin{enumerate}
\item[(1)] If  $\P(X\geq x)=\Theta(1) e^{-\lambda x^{\alpha}}$ as $x\rightarrow\infty$ with some $\lambda>0$ and $\alpha>1$, then for $a>0$, and for a sequence of integers $(t_n)$ such that $t_n=o(n^{\frac{1}{3}})$ and $t_n\rightarrow \infty$, we have
\[
\limsup_{n\rightarrow\infty}\frac{t_n^{\alpha-1}}{n^{\alpha/2}}\log \P\left(\sum_{i=1}^{t_n}X_i \geq a\sqrt{n}\right)\leq - \lambda a^\alpha.
\]
\item[(2)] If $\P(X\geq x)=\Theta(1) e^{-e^{x^{\alpha}}}$  as $x\rightarrow\infty$  with some $\alpha>0$, then for any $a>0$ and any sequence $t_n\uparrow \infty $ such that $t_n=o\left(\frac{\sqrt{n}}{(\log n)^{2/\alpha+1}}\right)$,
\[
\liminf_{n\rightarrow\infty}\frac{t_n^{\alpha}}{n^{{\alpha}/{2}}}\log\left[-\log \P\left(\sum_{i=1}^{t_n}X_i \geq a\sqrt{n}\right)\right]\geq a^{\alpha}.
\]
\end{enumerate}

\end{lem}
\begin{proof}
It suffices to consider $\P\left(\sum_{i=1}^{t_n}|X_i| \geq a\sqrt{n}\right)$. Observe that
\begin{equation}\label{keyA2}
\P\left(\sum_{i=1}^{t_n}|X_i| \geq a\sqrt{n}\right)\leq  \P\left(\sup_{1\leq i\leq t_n}|X_i|\geq n\right)+\P\left(\sum_{i=1}^{t_n}|X_i| \geq a\sqrt{n}, \sup_{1\leq i\leq t_n}|X_i|< n\right).
\end{equation}

{\it Proof of (1):} Note that there exists $c_{15}>0$ such that for any $x\geq0$,
\[
\P(|X|\geq x)\leq c_{15} e^{-\lambda x^{\alpha}}.
\]
Apparently,
\begin{eqnarray}\label{A2up01}
\P\left(\sup_{1\leq i\leq t_n}|X_i|\geq n\right)\leq t_n \P(|X|\geq n)\leq c_{14} t_n e^{-\lambda n^{\alpha}}.
\end{eqnarray}
Meanwhile,
\begin{eqnarray}\label{A2up}
\ar\ar\P\left(\sum_{i=1}^{t_n}|X_i| \geq a\sqrt{n}, \sup_{1\leq i\leq t_n}|X_i|< n\right)\cr
\ar\ar\quad=
\sum_{x_i\in [0,n)\cap {\mbb N}, i=1,\cdots, t_n}\P\left(\sum_{i=1}^{t_n}|X_i| \geq a\sqrt{n}, \sup_{1\leq i\leq t_n}|X_i|< n, |X_i|\in[x_i, x_i+1)\right)
\cr
\ar\ar\quad\leq
\sum_{\sum_{i=1}^{t_n}x_i\geq a\sqrt{n}-t_n; \atop x_i\in [0,n)\cap {\mbb N}, i=1,\cdots, t_n}
\P\left( |X_i|\in[x_i, x_i+1],\, \forall 1\leq i\leq t_n\right)
\cr\ar\ar\quad\leq \sum_{\sum_{i=1}^{t_n}x_i\geq a\sqrt{n}-t_n; \atop x_i\in [0,n)\cap {\mbb N}, i=1,\cdots, t_n} c_{15}^{t_n}\exp\l\{-\lz \sum_{i=1}^{t_n}x_i^{\alpha}\r\}
\cr\ar\ar\quad
\leq
\sum_{\sum_{i=1}^{t_n}x_i\geq a\sqrt{n}-t_n; \atop x_i\in [0,n)\cap {\mbb N}, i=1,\cdots, t_n} c_{15}^{t_n}\exp\l\{-\lz \frac{(a\sqrt{n}-t_n)^{\alpha}}{t_n^{\alpha-1}}\r\}\cr\ar\ar\quad
\leq(nc_{15})^{t_n}\exp\l\{-\lz \frac{(a\sqrt{n}-t_n)^{\alpha}}{t_n^{\alpha-1}}\r\},
\end{eqnarray}
where in the last inequality  we use the fact that
the convexity of mapping $x\mapsto x^\alpha$  for $\alpha>1$ implies
\[
\sum_{i=1}^{t_n}x_i^{\alpha}\geq t_n \left(\frac{\sum_{i=1}^{t_n}x_i}{t_n}\right)^\alpha \geq \frac{(a\sqrt{n}-t_n)^{\alpha}}{t_n^{\alpha-1}}.
\]
Plugging  \eqref{A2up01} and \eqref{A2up} into \eqref{keyA2}, we obtain that
\[
\limsup_{n\rightarrow\infty}\frac{t_n^{\alpha-1}}{n^{\alpha/2}}\log \P\left(\sum_{i=1}^{t_n}|X_i| \geq a\sqrt{n}\right)\leq - \lambda a^\alpha,
\]
which suffices to conclude (1) of Lemma \ref{A2}.\\
{\it Proof of (2):} Similarly,  for any $\ez>0$ there exists some constant $c_{16}\geq1$ such that
\begin{eqnarray}\label{A2up+}
\P\left(\sum_{i=1}^{t_n}|X_i| \geq a\sqrt{n}\right)\ar\leq \ar c_{15} t_n e^{-e^{n^\alpha}}+\P\left(\sum_{i=1}^{t_n}|X_i| \geq a\sqrt{n}, \sup_{1\leq i\leq t_n}|X_i|< n\right)\cr
\ar\leq \ar c_{15}t_n e^{-e^{n^\alpha}}+ (c_{16}n)^{t_n} \exp\l\{- (1-\ez)e^{\l(\frac{a\sqrt{n}}{t_n}\r)^\alpha}\r\},
\end{eqnarray}
where we use the fact
\[
\sum_{i=1}^{t_n}e^{x_i^\alpha}\geq \exp\l\{\max_{1\leq i\leq t_n}x_i^\alpha\r\}\geq \exp\l\{{\l(\frac{\sum_{i=1}^{t_n}x_i}{t_n}\r)^\alpha}\r\}.
\]
Consequently,
\[
\liminf_{n\rightarrow\infty}\frac{t_n^\alpha}{n^{\alpha/2}}\log\left[-\log \P\left(\sum_{i=1}^{t_n}|X_i| \geq a\sqrt{n}\right)\right]\geq a^\alpha.
\]
We have completed the proof.
\end{proof}

\bigskip\bigskip

\noindent{\small Xinxin Chen}

\noindent{Institut Camille Jordan, C.N.R.S. UMR 5208, Universite Claude Bernard Lyon 1, 69622 Villeurbanne Cedex, France.}

\noindent{E-mail: {\tt xchen@math.univ-lyon1.fr}}

\bigskip

\noindent{\small Hui He}

\noindent{School of Mathematical Sciences, Beijing Normal University,
Beijing 100875, People's Republic of China.}

\noindent{E-mail: {\tt hehui@bnu.edu.cn}}


\begin{thebibliography}{12}

%

\bibitem{As76}
S. Asmussen (1976):
Convegence rates for branching processes.
{\em Annals of Probablity} 4(1): 139--146.

\bibitem{AK76}
S. Asmussen and N. Kaplan (1976):
Branching random walk I.
{\em Stochastic Processes and Their Applications} 4: 1--13.

\bibitem{A94}
K. B. Athreya (1994):
Large Deviation Rates for branching processes-I. Single type case. {\em The Annals of Applied Probability} 4:779--790, 1994.

\bibitem{AN72}
K. B. Athreya and P. E. Ney (1972): {\em Branching Processes}, Springer, Berlin, 1972.


\bibitem{B90}
J. D. Biggins (1990):
The central limit theorem for the supercritical branching random walk, and related results.
{\em Stochastic Process. Appl.} 34: 255--274

\bibitem{C01}
X. Chen (2001): Exact convergence rates for the distribution of the particles in branching random walks. {\em Ann. Appl. Probab.} 11:1242--1262.



\bibitem{DZ98}
A. Dembo and O. Zeitouni (1998): {\em Large Deviation Techniques and Applications.} Springer-Verlag.

\bibitem{FW08}
K. Fleischmann and V. Wachtel (2008):
\newblock Large deviations for sums indexed by the generations of a Galton-Watson process.
\newblock {\em Probab. Theory and Related Fields} 141:445--470.

\bibitem{G17}
Z. Gao (2017):
\newblock
Exact convergence rate of the local limit theorem for branching random walks on the integer lattice.
To appear in \newblock {\em Stochastic Processes and their Applications}.

\bibitem{GL16}
Z. Gao and Q. Liu (2016):
\newblock Exact convergence rates in central limit theorems for a branching random walk with a random environment in time.
\newblock {\em Stochastic Processes and their Applications}  126:2634--2664.

\bibitem{GL17}
Z. Gao and Q. Liu (2017):
Second and third orders asymptotic
expansions for the distribution of particles in
a branching random walk with a random
environment in time.
To appear in {\em Bernoulli.}

\bibitem{Ha63}
T. E. Harris (1963):
{\em The theory of branching processes.}
Springer-Verlag.

\bibitem{He71}
C. C. Heyde (1971):
Some central limit analogues for supercritical Galton-Watson processes.
{\em Journal of Applied Probability} 8: 52--59.

\bibitem{HS07}
Y. Hu and Z. Shi (2007):
A subdiffusive behaviour of recurrent random walk in random environment on a regular tree.
{\em Probab. Theory. Relat. Fields} 138: 521--549.

\bibitem{HuShi09}
Y. Hu and Z. Shi (2009)
Minimal position and critical martingale convergence in branching random walks, and directed polymers on disordered trees.
{\em Ann. Prob.} 37: 403--813.

\bibitem{HS15}
Y. Hu and Z. Shi (2015):
The slow regime of randomly biased walks on trees.
{\em Ann. Prob.} 44: 3893--3933.

\bibitem{Ka82}
N. Kaplan (1982): A note on the branching random walk. {\em Journal of Applied Probability} 19(2):
421--424.

\bibitem{K82}
C.F. Klebaner (1982):
Branching random walk in varying environments.
{\em Adv. in Appl. Probab.} 14: 359--367.

\bibitem{Liu98}
Q. Liu (1998):
Fixed points of a generalised smoothing transformation and applications to branching processes.
{\em Adv. Appl. Prob.} 30: 85--112.

\bibitem{Liu06}
Q. Liu (2006):
On generalised multiplicative cascades.
{\em Stoch. Proc. Appl.} 86: 263--286.

\bibitem{LP15}
O. Louidor and W. Perkins (2015):
 \newblock Large deviations for the empirical distribution in the branching random walk. \newblock {\em Electronic Journal of Probability}, no. 18, 1--19.

 \bibitem{LT17}
 O. Louidor and E. Tsairi (2017):
 \newblock Large deviations for the empirical distribution in the general branching random walk.
 \newblock {\em arXiv:1704.02374}.

 \bibitem{Na79}
S. V. Nagaev (1979):
\newblock Large deviations of sums of independent random variables.
\newblock {\em Ann. Probab.} 7:745--789.

\bibitem{Ne86}
J. Neveu (1986):
Arbres et processus de Galton-Watson.
{\em Ann. Inst. Henri Poincar\'e Probab. Stat.} 22: 199--207

\bibitem{Re94}
P. R\'ev\'esz (1994):
Random walks of infinitely many particles.
{\em World Scientific Publishing Co. Inc. River Edge, NJ.}

\bibitem{S12}
Z. Shi (2015):
Branching random walks.
{\em \'Ecole d'\'Et\'e de Probabilit\'es de Saint-Flour XLII-2012}. Lecture Notes in Mathematics 2151. Springer, Berlin.

\end{thebibliography}
\end{document}